\documentclass[3p,sort&compress]{elsarticle}

\usepackage{amsmath,amssymb,amsfonts,mathptmx}
\usepackage{graphicx}
\usepackage{subfigure}
\usepackage{ulem}
\usepackage{tikz}

\usepackage{paralist}

\usepackage{subfigure}
\usepackage{float}
\usepackage{times}
\usepackage{tikz}
\usepackage{cases}
\usetikzlibrary{intersections}
\usepackage{xcolor}
\usepackage{bbm}

\usetikzlibrary{shadings}
\usetikzlibrary[intersections]
\usetikzlibrary{arrows,shapes}
\usetikzlibrary{shadings}
\usetikzlibrary[intersections]
\usetikzlibrary{arrows,shapes}

\newtheorem{theorem}{Theorem}
\newtheorem{lemma}{Lemma}[section]
\newtheorem{proposition}{Proposition}[section]

\newtheorem{remark}{Remark}[section]
\newdefinition{definition}{Definition}[section]
\newproof{proof}{Proof}
\newdefinition{assumption}{Assumption}[section]

\numberwithin{equation}{section}

\begin{document}
\begin{frontmatter}


\title{Global viscosity solutions to Lorentzian eikonal equation on globally hyperbolic space-times}


\author[SyZ]{Siyao Zhu}
\ead{zsydtc0807@gmail.com}

\author[HgW]{Hongguang Wu}
\ead{whg@cczu.edu.com}

\author[XC]{Xiaojun Cui\corref{cor}}
\ead{xcui@nju.edu.cn}

\address[SyZ]{Department of Mathematics, Nanjing University, Nanjing 210093, P.R. China}

\address[HgW]{Department of Mathematics, Changzhou University, Changzhou 213164, P.R. China}

\address[XC]{Department of Mathematics, Nanjing University, Nanjing 210093, P.R. China}

\cortext[cor]{Corresponding author.}

\journal{}

\begin{abstract}
In this paper, we show that any globally hyperbolic space-time admits at least one globally defined locally semiconcave function, which is a viscosity solution to the Lorentzian eikonal equation. According to whether the time orientation is changed, we divide the set of viscosity solutions into some subclasses. We show if the time orientation is consistent, then a viscosity solution has a variational representation locally. As a result, such a viscosity solution is locally semiconcave and has some weak KAM properties, as the one in the Riemannian case. On the other hand, if the time orientation of a viscosity solution is non-consistent, it will exhibit some peculiar properties which make this kind of viscosity solutions are totally different from the ones in the Riemannian case.
\end{abstract}

\begin{keyword}
Global viscosity solution \sep
Lorentzian eikonal equation \sep
weak KAM theory
\end{keyword}
\end{frontmatter}

\section{Introduction}

  Since the landmark work of Penrose \cite{Penrose1964, Penrose1965a}, the study of Lorentzian causality theory has attracted many researchers in recent decades. In this theory, Lorentzian geodesic flow is a central research object. Eikonal equation, which can be regarded as a special Hamilton-Jacobi equation, plays an important role in the study of Lorentzian geodesic flow. Actually, if the eikonal equation admits a smooth solution, the integral curves to the gradient field are geodesics, e.g. \cite[Lemma B.1]{Beem 1996}. Unfortunately, smooth solutions for such partial differential equations are usually absent. At this moment, viscosity solutions introduced by Crandall and Lions \cite{Lions1983TAMS} show powerful advantages than other weak solutions (e.g., Viscosity solutions usually have better well-posedness in Dirichlet problem of partial differential equations).

  On the other hand, a field closely related to viscosity solutions to the Hamilton-Jacobi equation induced by a Tonelli Lagrangian is the so-called weak KAM theory. Weak KAM theory, introduced by Fathi \cite{FathiWKAM}, establishes a connection between Aubry-Mather theory for a Tonelli Lagrangian and viscosity solution theory of the associated Hamilton-Jacobi equation. Informally, backward action-minimizing curves, the main object in Aubry-Mather theory, are exactly calibrated curves of some viscosity solutions. In other words, viscosity solutions are nothing but the weak KAM solutions of backward type. This can give an explicit variational representation for global viscosity solutions to the Hamilton-Jacobi equation.  Such a variational representation can provide us a geometrical intuition for the gradient dynamics of the viscosity solutions.  For some recent works on weak KAM theory, Lorentzian geodesic flows and Lorentzian Aubry-Mather theory, we refer interested readers to \cite{CuiIJM2016, Beem 1996, CuiJMP2014, JinGD2018, ERDGTG2003, Flores Mams 2013, SuhrJFTA2019, SuhrMM2022} and the references therein.
  Inspired by these works, it is interesting to construct a space-time counterpart of weak KAM theory for viscosity solutions to the eikonal equation. However, classical weak KAM theory depends heavily on the positive-definiteness of the Lagrangian. On space-time, non-positive-definiteness of Lorentzian metric makes problem complicated and challenging. Furthermore, non-completeness is a common phenomenon on a space-time, it makes things worse.

  As far as we know, there are only a few works on this topic.  Cui and Jin \cite{CuiJMP2014} study space-time with the regular cosmologcial time function $\tau$ and show that $-\tau$ is a viscosity solution to Lorenzian eikonal equation. Jin and Cui \cite{JinGD2018} establish the existence of global viscosity solution to Lorentzian eikonal equation on the Abelian cover of a class A Lorentzian 2-torus. In \cite{SuhrCMP2018}, Bernard and Suhr propose a cone field theory for generalized space-time and establish the equivalence between global hyperbolicity and the existence of a steep temporal function, where the steep temproal function is a smooth (of course, viscosity) subsolution to Lorentzian eikonal equation in our setting. It is not surprising that viscosity solutions on a Riemannian manifold and space-time may have different properties. For example, in the Riemannian case, viscosity solutions are always locally semiconcave, but even on the 2-dimensional Minkowski space-time ($\mathbb{R}^2, dy^2-dx^2$), a viscosity solution may not be locally semiconcave. Actually, it is easy to check that $f(x,y)=|x|$ is a viscosity solution to the eikonal equation on the 2-dimensional Minkowski space-time. Obviously, $f$ is not locally semiconcave in any neighbourhood of $x=0$. In this paper, we will give a much more complete theory on viscosity solutions of eikonal equation on a globally hyperbolic space-time. More precisely, we classify the set of viscosity solutions to the Lorentzian eikonal equation according to whether the time orientation of the solution is consistent. We show that a viscosity solution has a (local) variational representation when the time orientation is consistent. Such a variational representation ensures that the viscosity solution has the similar properties as in the Riemannian case, for instance, local semiconcavity and some properties of weak KAM type. We also study the properties of viscosity solution when the time orientation is non-consistent. In this case, viscosity solutions exhibit many peculiar properties, which are totally different from the ones in the Riemannian case.

\section{Preliminaries and main results}
  We will use classical terminology as in \cite{Beem 1996, Flores Mams 2013, CuiIJM2016, CuiJMP2014, JinGD2018}. Let $(M, g)$ be a space-time (i.e., a connected, time-oriented, smooth Lorentzian manifold with  the Lorentzian metric $g$), where the signature for the Lorentzian metric is $(-,+, \cdots , +)$. A point $p\in M$ is usually called an event from the viewpoint of general relativity. For each $p\in M$, let $T_p M$ be the tangent space at $p$ and $TM$ be the tangent bundle of $M$. A tangent vector $V \in T_p M$ is called timelike, spacelike or lightlike if $g(V, V)<0$, $g(V, V)>0$ or $g(V, V)=0$ ($V\neq 0$), respectively. Also, a tangent vector $V\in T_p M$ is called causal if $V$ is either timelike or lightlike and is called non-spacelike if $g(V, V)\leq 0$. The set of null vectors consists of the zero vector and lightlike vectors. The time-orientation of $(M,g)$ ensures that $M$ admits a continuous, nowhere vanishing, timelike vector field $X$. $X$ is used to separate the causal vectors at each base point into two classes which are called past-directed and future-directed respectively. A causal tangent vector $V\in T_p M$ is said to be past-directed (respectively, future-directed) if $g(X(p), V)>0$ (respectively, $g(X(p), V)<0$). Then a continuous piecewise $C^1$ curve $\gamma: I\rightarrow M$ ($I\subset \mathbb{R}$ is an interval of the real line) is called causal, timelike, lightlike and past-directed or future-directed if the tangent vector $\dot{\gamma}(s)$ is causal, timelike, lightlike and past-directed or future-directed at every differentiable point $\gamma(s)$.

    For two events $p, q\in M$, if there exists a future-directed timelike (respectively, future-directed causal or constant) curve from $p$ to $q$, we say $p, q$ are chronologically related (respectively, causally related) and denote by $p\ll q$ (respectively, $p\leq q$). For a given $p\in M$, define chronological past $I^-(p)$, chronological future $I^+(p)$, causal past $J^-(p)$, causal future $J^+(p)$ as follows:
\begin{align*}
I^-(p)=\{q\in M: q\ll p\},\\
I^+(p)=\{q\in M: p\ll q\},\\
J^-(p)=\{q\in M: q\leq p\},\\
J^+(p)=\{q\in M: p\leq q\}.
\end{align*}
Analogously, for subsets $A\subseteq M$, we can define $I^-[A]$, $I^+[A]$, $J^-[A]$ and $J^+[A]$. For example, $I^-[A]:=\{q\in M| q\ll s ~~\text{for some}~~ s\in A\}$. We call a subset $A$ of $M$ is an achronal (respectively, acausal) set if no two points of $A$ are chronologically related (respectively, causally related ).

\begin{definition}[\protect{\cite{Beem 1996, BernalCQG2007}}]\label{Def10}
 A space-time $(M, g)$ is said to be causal if there are no causal loops. A space-time $(M, g)$ is globally hyperbolic if it is causal and the sets $J^+(p)\cap J^-(q)$ are compact for all $p, q\in M$.
\end{definition}

Throughout this paper, we always assume that $(M, g)$ is a globally hyperbolic space-time. We are going to use two metrics on $M$, one is the Lorentzian metric $g$, the other is an auxiliary complete Riemannian metric $h$. The length functionals associated to $g$ and $h$ are denoted by $L(\cdot)$ and $L_h(\cdot)$ respectively. Let $\Omega_{x,y}$ denote the path space of all future-directed causal curves and constant curves $\gamma:[0, 1]\rightarrow M$ with $\gamma(0)=x$ and $\gamma(1)=y$. For a causal curve $\gamma\in \Omega_{x,y}$, the length of $\gamma$ is given by
\begin{align*}
L(\gamma):=\int_0^1 \sqrt{-g(\dot{\gamma}(s), \dot{\gamma}(s))}ds.
\end{align*}

Correspondingly the distance functions associated to $g$ and $h$ are denoted by $d(\cdot, \cdot)$ and $d_h(\cdot, \cdot)$ respectively. More precisely, for any $x, y\in M$ with $x\leq y$, $d(x, y)$ is given by
\begin{align*}
d(x, y)=\sup\limits_{\gamma}\{L(\gamma):\gamma \in \Omega_{x, y}\}.
\end{align*}

The norms associated to $g$ and $h$ are denoted by $|\cdot|$ and $|\cdot|_h$ respectively. Recall that for a non-spacelike vector $V\in TM$,  $|V|=\sqrt{-g(V,V)}$. Let $\nabla$ (respectively, $\nabla_h$) be the gradient induced by the Lorentzian metric $g$ (respectively, Riemannian metric $h$). There is a timelike  Lorentzian eikonal equation associated to the Lorentzian metric $g$:
\begin{equation}\label{E1}
                 g(\nabla u, \nabla u)= -1.
\end{equation}
In this paper, we devote ourself to studying viscosity solutions of equation \eqref{E1}. Before stating our results, we need to introduce some more definitions for convenience. The next two definitions are about to generalized gradients and viscosity solutions to eikonal equation (\ref{E1}).

\begin{definition}[\protect{\cite[Definition 3.1.6]{CannarsaPNDE2004}}]\label{Def3}
Let $f: M\rightarrow \mathbb{R}$ be a continuous function on $(M, g)$. A vector $V \in T_pM$ is called to be a subgradient (resp., supergradient) of $f$ at $p\in M$, if there exists a neighborhood $O$ of $p$ and a $C^1$ function $\phi: O \rightarrow R$ such that $f(p)=\phi(p)$, $\phi(x) \leq f(x)$ (resp., $\phi(x) \geq f(x)$) for every $x \in O$ and $\nabla \phi(p)= V$.
\end{definition}

  We denote by $\nabla^- f(p)$ the set of subgradients of $f$ at $p$, $\nabla^+ f(p)$ the set of supergradients of $f$ at $p$. With the help of $\nabla^- f(p)$ and $\nabla^+ f(p)$, we can state that the definition of viscosity solution as follows:

\begin{definition}\label{Def4}
A continuous function $f$ is called a viscosity subsolution of equation (\ref{E1}) if for any $p\in M$,
\[
             g(V, V)\leq -1 ~~\mbox{for every}~~ V\in \nabla^+ f(p).
\]
Similarly, $f$ is called a viscosity supersolution of equation (\ref{E1}) if for any $p\in M$,
\[
             g(V, V)\geq -1 ~~\mbox{for every}~~ V\in \nabla^- f(p).
\]

A continuous function $f$ is a viscosity solution of equation (\ref{E1}) if it is a viscosity subsolution and a viscosity
supersolution simultaneously.
\end{definition}

Let $Lip_{loc}(M)$ be the set of locally Lipschitz (with respect to the Riemannian metric $h$) functions on $M$ and
\[
\mathcal{S}(M):=\{f\in Lip_{loc}(M)| f \mbox{ is global viscosity solution to equation (\ref{E1})}\} .
\]

  In this paper, $\mathcal{S}(M)$ is the main object that we want to study systematically. Concretely, we expect to get a classification of $\mathcal{S}(M)$ and give a variational representation of elements in $\mathcal{S}(M)$. For this purpose, we need the following definitions.

\begin{definition}[\protect{\cite[Definition 2.8]{CuiJMP2014}}]\label{Def6}
Let $\phi: M\rightarrow \mathbb{R}$ be a locally Lipschitz function defined on the space-time $(M, g)$, then a vector $V\in T_p M$ is called a limiting gradient if there exists a sequence $p_k \subset M\setminus\{p\}$ with $\lim\limits_{k\rightarrow\infty}p_k=p$ such that $\phi$ is differentiable at $p_k$ for each $k\in\mathbb{N}$, and $\lim\limits_{k\rightarrow\infty} \nabla \phi(p_k)=V$. Here, the first limit is taken in the sense of manifold topology on $M$; the second limit is taken in the sense of any fixed chart that contains $p$. Since the first limit is taken, we know that when $k$ is sufficiently large, $p_k$ goes into that chart. The second limit does not depend on the choice of chart. We denote $\nabla^*\phi(p)$ to be the set of all limiting gradients of $\phi$ at $p$.
\end{definition}

   It should be mentioned that for any $f\in \mathcal{S}(M)$, if there exists a sequence $\{x_i\}_i\subseteq M\backslash \{x\}$ such that $x_i\rightarrow x$ as $i\rightarrow\infty$ and $f$ is differentiable at each $x_i$, then for sufficiently large $i$, $|\nabla_h f(x_i)|_h$ is bounded, where the bound depends on the Lipschitz constant of $f$ near $x$. Let $\kappa$ be the map such that $\kappa (\nabla_h g(\cdot))=\nabla g(\cdot)$ for any smooth function $g$ on $M$. Obviously, $\kappa$ is a linear transformation on each fiber of $TM$ and $|\kappa|_h$ is locally bounded, where $|\kappa|_h=\sup\limits_{V\in TM, V\neq 0}\frac{|\kappa (V)|_h}{|V|_h}$ is the operator norm with respect to $h$. Thus, the set $\{\nabla f(x_i)\}_i$ is contained in a compact subset of $TM$. More precisely, $\lim\limits_{i\rightarrow\infty}\nabla f(x_i)\in \nabla^*f(x)$.


\begin{definition}\label{Def9}
For a Lipschitz continuous function $f$ on $(M, g)$ where $\nabla f$ is timelike as long as it exists,
\begin{itemize}
\item We say the time orientation of $f$ is future-directed at $x$ if any limiting gradient $V\in \nabla^* f(x)$ is  future-directed timelike, the time orientation of $f$ is past-directed at $x$ if any limiting gradient $V\in \nabla^* f(x)$ is past-directed timelike.
\item We say the time orientation of $f$ changes at $x$ if there exist two timelike vectors $V, W$ in $\nabla^* f(x)$ such that $V$ is past-directed and $W$ is future-directed.
\item We say the time orientation of $f$ is consistent if for any two points $x, y\in M$, the vectors in both $\nabla^*f(x)$ and $\nabla^*f(y)$ have the same time orientation. Otherwise, we say the time orientation of $f$ is non-consistent.
\end{itemize}
\end{definition}

\begin{definition}[Forward calibrated curve]\label{Def12}
For any $u\in\mathcal{S}(M)$, a future-directed timelike curve $\alpha:[0,t]\rightarrow M$ ($t>0$) is said to be a forward calibrated curve of $u$ if $u(\alpha(0))=u(\alpha(s))-s$ for any $s\in [0,t]$.
\end{definition}

\begin{definition}[Partial Cauchy Surface]\label{Def10}
Let $A\subseteq M$ be an achronal set, the $\text{edge}(A)$ consists of all points $p\in \bar{A}$, such that every neighborhood $U$ of $p$ contains a timelike curve form $I^-(p, U)$ to $I^+(p, U)$ which does not meet $A$. Where $I^-(p, U)$ (respectively, $I^+(p, U)$) denotes the chronological past (respectively, future) of $p$ with respect to the space-time $(U, g|_U)$. The set A is said to be edgeless if $\text{edge}(A)=\emptyset$. A subset $S$ of $(M, g)$ is said to be a partial Cauchy surface if $S$ is acausal and edgeless in $M$.
\end{definition}


Here and in the following, for any $u\in \mathcal{S}(M)$ and $s\in \mathbb{R}$, the level set of $u$ and the  image set of $u$ are denoted by $u_s$ and $\text{Image}(u)$ respectively.  In general, for $s\in \text{Image}(u)$ $u_s$ is not a Cauchy surface even if $M$ is globally hyperbolic. We left such an example in Section \ref{sec:level sets}.

Based on the concepts mentioned above, our main results of this paper are stated as follows.

\begin{theorem}\label{the10}
  If $(M, g)$ is a globally hyperbolic space-time, then $\mathcal{S}(M)\neq \emptyset$. In other words, the eikonal equation admits at least one globally defined locally Lipschitz viscosity solution.
\end{theorem}

In the following we will propose some properties of $u\in\mathcal{S}(M)$ when  the time orientation of $u$ is consistent. When the time orientation of $u$ is always past-directed, we have

\begin{theorem}\label{the7}
 For any $u\in \mathcal{S}(M)$ for which the time orientation of $u$ at any point is past-directed, we have 
\begin{itemize}
  \item For each $s\in \text{Image}(u)$, $u_s$ is a partial Cauchy surface.
  \item There exists a $t_0>0$, which is dependent on $x$, such that for any $0<t<t_0$, 
\begin{equation}\label{25}
           u(x)=\inf\limits_{x\leq y, d(x,y)=t}u(y)-t.
\end{equation}
  \item The function $u$ is locally semiconcave on $(M, g)$.
\end{itemize}
\end{theorem}

 Actually, under the same conditions of Theorem \ref{the7}, viscosity solutions possess some  properties of weak KAM type like the ones in the Riemannian case. To state these results, recall that a future-directed, future inextendible timelike geodesic $\gamma:[0, T)\rightarrow M$ is said to be a future-directed timelike ray if $d(\gamma(0), \gamma(t))=L(\gamma)|_{[0,t]}$ for all $0\leq t \leq T$. About timelike rays on space-time, interested readers can refer to \cite[Definition 8.8]{Beem 1996}.

\begin{theorem}\label{the2}
 For any $u\in \mathcal{S}(M)$ for which the time orientation of $u$ at any point is past-directed, we have
\begin{itemize}
  \item  $u(x)$ is differentiable at $x\in M$ if and only if there exists a unique future-directed timelike ray $\gamma_x: [0, T)\rightarrow M$ that satisfies
\begin{align}\label{E15}
          \gamma_x(0)=x,~~ u(\gamma_x(t))=u(x)-t~~\mbox{for every } t\in [0 ,T)
\end{align}
and
\begin{align}\label{E16}
         \dot{\gamma_x}(0)=-\nabla u(x)
\end{align}
where $T$ is the maximal existence time of $\gamma_x$.
  \item For any $x\in M$, $u$ is differentiable at any $\gamma_x(t)$ for $t\in (0, T)$ and $\dot{\gamma_x}(t)=-\nabla u(\gamma_x(t))$.
\end{itemize}
\end{theorem}

As an application of Theorem \ref{the7}, \ref{the2}, we can draw the following conclusion

\begin{theorem}\label{prop1}
Let $\phi$, $\varphi\in \mathcal{S}(M)$ and the time orientation of each of them is consistent, then we have
\begin{itemize}
  \item If the time orientation of $\phi$ and $\varphi$ are identical, then $\min\{\phi , \varphi\}\in\mathcal{S}(M)$.
  \item If  the time orientation  of $\phi$ and $\varphi$ are different, then $\max\{\phi , \varphi\}\in \mathcal{S}(M)$ .
\end{itemize}
\end{theorem}

The results of the above theorems are based on the fact that the time orientation of $u$ is consistent. One may wonder what happens when the time orientation of $u$ is non-consistent. Because of distinctiveness of causal structure and non-positive-definiteness of the Lorentzian metric, the properties of viscosity solutions to the equation (\ref{E1}) are significantly different from the ones in the Riemannian case. We will propose an example in Section \ref{sec:6} to show that in this case the viscosity solutions lose many properties such as locally semiconcavity. For any $u\in \mathcal{S}(M)$, we introduce  the changing set $C_u$ by
\[
C_u:=\{x\in M|~ \text{there are} ~V, W \in \nabla^* u(x)~\text{such that}~ V ~\text{is a past-directed causal vector},~ W~\text{is a future-directed causal vector}\}.
\]


The following theorem characterizes the properties of $C_u$.
\begin{theorem}\label{the5}
 For any $u\in \mathcal{S}(M)$, we have
\begin{itemize}
  \item $C_u$ is a closed set without interior point. If $C_u\neq \emptyset$, then $M\setminus C_u$ contains at least two connected components.
  \item Let $\mathcal{M}$ be the  Lebesgue measure on $M$, then $\mathcal{M}(C_u)=0$.
  \item For any $x\in M\setminus C_u$, there exists a neighborhood $U_x$ of $x$ such that for the space-time $(U_x, g|_{U_x})$, Theorems \ref{the7}, \ref{the2} hold true for $u$.
  \item If $C_u$ is an acausal set, then $C_u$ is a partial Cauchy surface.
\end{itemize} 
 
\end{theorem}


The remainder of this paper is organized as follows. In Section \ref{sec:2}, we construct a function $u^+$ and prove $u^+\in \mathcal{S}(M)$, which completes the proof of theorem \ref{the10}. In Section \ref{sec:level sets}, we study the properties of level sets of $u$ when the time orientation of $u$ is consistent. In Section \ref{sec:uuu}, we prove $u$ has a variational representation when the time orientation of $u$ is consistent and as a consequence, $u$ is locally semiconcave. Based on the results in these two sections, Theorem \ref{the7} is proved. In Section \ref{sec:3}, we establish the weak KAM properties of  $u$ when  the time orientation of $u$ is consistent. Theorems \ref{the2}, \ref{prop1} are proved in this section. We prove Theorem \ref{the5} in Section \ref{sec:6}.

\section{Construction of global viscosity solutions to equation (\ref{E1})}\label{sec:2}
In this section, we firstly construct two globally defined functions $u^+$ and $u^-$ by a steep time function $\tau$ introduced in \cite{SuhrCMP2018}. Secondly, we will prove the well-posedness
and the semiconcavity of $u^+$ by using the method of support function. Finally, we will show that $u^+$ is a viscosity solution to equation (\ref{E1}), which completes the proof of Theorem \ref{the10}.

\subsection{The construction of $u^{\pm}$}

Recall that $(M, g)$ is globally hyperbolic, $d(\cdot, \cdot)$ is continuous with respect to both two variables and satisfies the reverse triangle inequality:
\[
       d(p, q)+d(q, r)\leq d(p, r)  ~~~  \text{for}~~ p,q,r\in M ~~\text{and}~~ p\leq q\leq r.
\]

Let $\gamma:[a, b)\rightarrow M$ be a past-directed (respectively, future-directed) causal curve. The point $p\in M$ is said to be a past (respectively, future) endpoint of $\gamma$ corresponding to $t=b$ if
\[
\lim\limits_{t\rightarrow b^-}\gamma(t)=p.
\]

A causal curve is said to be past (respectively, future) inextendible if it has no past (respectively, future) endpoint. A causal curve $\gamma:(a, b)\rightarrow M$ is said to be inextendible if it is both past and future inextendible.

\begin{definition}[Cauchy Surface]\label{Def8}
 A Cauchy surface $S$ is a subset of $M$ for which every inextendible causal curve intersects exactly once.
\end{definition}

\begin{definition}[Cauchy development]\label{Def13}
Given a closed subset $S$ of $(M, g)$, the past Cauchy development or domain of dependence of $S$, $D^-(S)$, is defined as the set of all points $q$ such that every future inextendible causal curve from $q$ intersects $S$. The future domain of dependence of $S$, $D^+(S)$, is defined in a similar way by exchanging the roles of the future and the past. The Cauchy development or the domain of dependence is the set $D(S)=D^-(S)\cup D^+(S)$.
\end{definition}

By \cite[Theorem 3]{SuhrCMP2018}, there exists a steep temporal function $\tau : M \rightarrow R$ such that $\tau$ is strictly increasing along any future-directed causal curve and
\begin{equation}\label{E24}
|d \tau (V)|\geq \max\{\sqrt{-g(V, V)} , |V|_h\}
\end{equation}
for all timelike vector $V \in TM$. \cite[Corollary 1.8]{SuhrCMP2018} implies that there exists a diffeomorphism $M\cong R\times N$ such that
\[
       \tau: M\cong R\times N\rightarrow R,~~p\cong (s, x)\rightarrow s.
\]

  For $s\in \mathbb{R}$, let $\tau_s$ be the level set of $\tau$, i.e., $\tau_s:=\{p\in M| \tau(p)=s\}$.
\begin{remark} 	
For an inextendible causal curve $\gamma: (a , b)\rightarrow M$, by the completeness of $h$, we have $\lim\limits_{t\rightarrow a^+}\tau (\gamma(t))=-\infty$ and $\lim\limits_{t\rightarrow b^-}\tau (\gamma(t))=\infty$.
\end{remark}

   For $s\in \mathbb{R}$, each $\tau_s$ is a Cauchy surface. For any $p\in I^-[\tau_s]$, let $d(p, \tau_s):=\sup\limits_{z\in \tau_s} d(p,z)$. Similarly, let $d(\tau_s, p):=\sup\limits_{z\in \tau_s} d(z,p)~\text{for any}~ p\in I^+[\tau_s]$.
\begin{remark} 	
For any $s\in \mathbb{R}$ and any $p \in I^-[\tau_s]$, $d(p, \tau_s)$ is finite. Indeed, by (\ref{E24}), for any future-directed causal curve $\gamma$ connected $p$ to $\tau_s$, $\sqrt{-g(\dot{\gamma}, \dot{\gamma})} \leq  d \tau (\dot{\gamma})$. Integrating on both sides,  we have $L(\gamma)\leq s-\tau(p)$. Taking supremum, we get $d(p, \tau_s)\leq s- \tau(p)$.
\end{remark}

\begin{definition}\label{Def1}
For fixed $s>0$, $x_0 \in \tau_0$ and for any $p\in I^-[\tau_s]$, define:
\begin{itemize}
\item $u^+_{\tau_s}(p):= d(x_0, \tau_s)-d(p,\tau_s)=\sup\limits_{z\in \tau_s} d(x_0,z)-\sup\limits_{z\in \tau_s} d(p,z)$,
\item $u^+(p):=\lim\limits_{s\rightarrow +\infty}u^+_{\tau_s}(p)$.
\end{itemize}
Similarly, for fixed $s<0$, $x_1 \in  \tau_0$ and for any $p\in I^+[\tau_s]$, define:
\begin{itemize}
\item$ u^-_{\tau_s}(p):= d(\tau_s,x_1)-d(\tau_s,p)= \sup\limits_{z\in \tau_s} d(z,x_1)-\sup\limits_{z\in \tau_s} d(z,p)$,
\item $u^-(p):=\lim\limits_{s\rightarrow -\infty}u^-_{\tau_s}(p)$.
\end{itemize}
\end{definition}
 By the definition of $u^+_{\tau_s}(p)$, the continuity of $u^+_{\tau_s}(p)$ can be derived from the continuity of $d(p, \tau_s)$. Obviously, $d(p,\tau_s)$ is a lower semicontinuous function on $I^-[\tau_s]$ since it is the supremum of a family of continuous functions. To get the continuity of $d(p,\tau_s)$, we only need to prove the upper semicontinuity. For this purpose, we introduce the following two lemmas. Throughout this paper, for a subset $A\subseteq M$, we use $\text{int} (A)$ to denote the set of all the interior points of $A$.

\begin{lemma}[\protect{\cite[Lemma 14.40]{Oneill1983}}]\label{lem25}
 Let $\Sigma$ be an achronal set. If $x\in \textit {int}  D(\Sigma)\setminus I^+(\Sigma)$, then $J^+(x)\cap D^-(\Sigma)$ is compact.
\end{lemma}

\begin{remark}\label{remark1}
 For any $x\in D^-(\Sigma)$, we set $I^+_{\Sigma}(x)=\Sigma\cap I^+(x)$, $J^+_{\Sigma}(x)=\Sigma\cap J^+(x)$. Then $J^+_{\Sigma}(x)=\overline{I^+_{\Sigma}(x)}$ is compact if $D(\Sigma)$ is an open subset of $M$ by Lemma \ref{lem25}. Moreover, if $\Sigma$ is a Cauchy surface, there exists a future-directed maximal timelike geodesic segment $\gamma_{x, \Sigma}:[0, d(x, \Sigma)]\rightarrow M$ connected $x$ to $\Sigma$ such that $\gamma_{x, \Sigma}(0)=x$ and $d(x, \Sigma)= L(\gamma_{x, \Sigma})$.
\end{remark}

\begin{lemma}[\protect{\cite[Lemma 3.2]{EschenburgJDG1988}}]\label{lem10}
Let $U$ be a causal convex neighborhood of $p$. If $U$ is small enough, there is a constant $C>0$, $T>0$ with the following property: If $p\in U$ and $s>2T$ then every maximal unit speed geodesic segment $\alpha$ from $p$ to $\tau_s$ satisfies $\|\alpha'(0)\|_h\leq C$.
\end{lemma}

Lemma \ref{lem10} shows that the set of initial tangent vectors of these maximal geodesic segments stays in a compact subset of tangent bundle $TU$.

Now suppose $d(p,\tau_s)$ is not upper semicontinuous at $p$ in $U$. Then there exists a sequence $p_n \in U$ with $p_n\rightarrow p$ but $\lim\limits_{n\rightarrow\infty}d(p_n, \tau_s)>d(p, \tau_s)$. Thanks to Remark \ref{remark1}, let $\alpha_n$ be a maximal timelike geodesic segment from $p_n$ to $\tau_s$ and reparametrize it by $h$-arc length. By the Limit Curve Lemma \cite[Lemma 14.2]{Beem 1996} and Lemma \ref{lem10}, there exists a subsequence ${\alpha_{n_k}}$ which converges to a timelike curve $\alpha$ from $p$ to $\tau_s$, so we have
\[
d(p_n, \tau_s)=L(\alpha_{n_k})\rightarrow L(\alpha)\leq d(p, \tau_s).
\]
This contradiction implies the continuity of $u^+_{\tau_s}(p)$ on $U$.

\subsection{A priori estimate of $u^+$}

  In this subsection we will show that both $u^+$ and $u^-$ are well-defined. Actually, we only need to prove that $u^+(p)$ is well-defined and $u^-(p)$ can be treated analogously. By the Arzela-Ascoli theorem and $u^+_{\tau_s}(x_0)=0$, it is sufficient to show that $u^+_{\tau_s}$ are locally equi-Lipschitz.
\begin{lemma}\label{lem1}
For $p, q\in I^-[\tau_s]$ with $p\leq q$, the function $u^+_{\tau_s}$ satisfies
\[
    u^+_{\tau_s}(q)- u^+_{\tau_s}(p)\geq d(p, q).
\]
\end{lemma}
\begin{proof}
By the reverse triangle inequality of Lorentzian distance function $d$, for any $q\leq z\in \tau_s$, we have
\[
d(p, q)+d(q, z)\leq d(p,z).
\]
Then by the definition of $u^+_{\tau_s}$, it is easy to see that
\begin{align*}
 &u^+_{\tau_s}(q)-d(p,q)\\
=&d(x_0, \tau_s)-d(q, \tau_s)-d(p,q)\\
=&                 d(x_0, \tau_s)-\sup\limits_{z\in \tau_s} d(q,z)-d(p,q)\\
\geq & d(x_0, \tau_s)-\sup\limits_{z\in \tau_s} d(p,z)\\
=&u^+_{\tau_s}(p).
\end{align*}\qed
\end{proof}

\begin{proposition}\label{the1}
$u^+_{\tau_s}(x)$ are locally equi-Lipschitz functions, and there exists a subsequence $s_n$ and a Lipschitz function $u^+(x)$ such that
\[
\lim\limits_{s_n \rightarrow +\infty} u^+_{\tau_{s_n}}(x)=u^+(x) ~~\text{for any} ~~x\in M.
\]
Moreover, $u^+(x)$ is a locally semiconcave function on $M$. In particular, when $s_n \rightarrow \infty$, $I^-[\tau_{s_n}] \rightarrow M$.
\end{proposition}

In order to prove Proposition \ref{the1}, we need to introduce the concept of support function.

\begin{definition}\label{Def2}
 For each $s>0$, $p\in I^-[\tau_s]$, there exists $q_s\in \tau_s$ and a future-directed timelike geodesic $\gamma$ such that $\gamma(0)=p$, $\gamma(d(p, q_s))=q_s$ and $d(p, q_s)= d(p,\tau_s)=L(\gamma)$. For any fixed $p_{1,s}\in \gamma$, we choose a suit neighborhood $O$ of $p$ such that $O\subseteq I^-(p_{1,s})$. Define
\begin{align*}
u^+_{p, s}: O \rightarrow [-\infty, +\infty]
\end{align*}
with
\begin{align*}
u^+_{p, s}(x)= d(x_0,\tau_s)-d(x,p_{1,s})-d(p_{1,s},q)~~\text{for any}~~x\in O.
\end{align*}
\end{definition}

Readers may wonder about the dependence of  $u^+_{p, s}$ on $p_{1, s}$ here, but by the following lemma, $p_{1, s}$ can be chosen in a compact subset of $M$. Therefore, the choice of $p_{1, s}$ does not prevent the equi-Lipschtiz property of $u^+_{p, s}$.

\begin{lemma}\label{lem24}
There exists a subsequence $s_n$ such that for each $s_n$ the $p_{1,s_n}$ in Definition \ref{Def2} can be chosen in a compact subset of $M$ and the neighborhood $O$ in Definition \ref{Def2} does not depend on $s_n$ for all $s_n$ large enough.
\end{lemma}
\begin{proof}
For $s$ large enough and any $x\in I^-[\tau_s]$, we use $\gamma_{x, s}$ to denote a future-directed maximal geodesic segment between $x$ to $\tau_s$. For each $s$, we extend $\gamma_{x, s}$ to be a future-directed inextendible timelike curve and reparameterize it by $h$-arc length. These reparametrzed curves are denoted by $\tilde{\gamma}_{x, s}$.  By the fact that $\tilde{\gamma}_{x, s}(0)=\gamma_{x, s}(0)=x$ and Limit Curve Lemma \cite[Lemma 14.2]{Beem 1996}, there exists a subsequence $s_n$ and a future-directed inextendible causal curve $\tilde{\gamma}_x: [0, \infty)\rightarrow M$ such that $\tilde{\gamma}_{x, s_n}$ converges to $\tilde{\gamma}_x$ uniformly on any compact subset of $[0, \infty)$.  We choose a constant $T>0$ and a neighborhood $O$ of $x$ such that $\bar{O}\subseteq I^-(\tilde{\gamma}_x(T))$. Since $I^-$ is inner continuous \cite[p. 59]{Beem 1996}, there exists a neighborhood $U$ of $I^-(\tilde{\gamma}_x(T))$ such that $\bar{O} \subseteq I^-(q)$ for any $q\in U$.  We define $p_{1,s_n}:= \tilde{\gamma}_{x, s_n}(T)$. Since $\tilde{\gamma}_{x, s_n}(T)$ converges to $\tilde{\gamma}_x(T)$, for sufficiently large $s_n$, the sequence $p_{1,s_n}$ always stays in a compact neighborhood of $\tilde{\gamma}_x(T)$ and the neighborhood $O$ is a subset of  $I^-(\tilde{\gamma}_{x, s_n}(T))$. \qed
\end{proof}

  Due to the reverse triangle inequality, it is easy to see that for any $s>0$, $u^+_{p, s}$ given by Definition \ref{Def2} is a continuous upper support function for $u^+_{\tau_s}(x)$ at $p$, i.e., $u^+_{p, s}(x)\geq u^+_{\tau_s}(x)$ for all $x$ near $p$ with equality when $x=p$. With the help of $u^+_{p, s}(x)$, we can prove that $u^+_{\tau_s}(x)$ are locally equi-Lipschitz.

Let $d_s(x)= d(x, \tau_s)$, it is sufficient to show that the functions $u^+_{\tau_s}$ are Lipschitz continuous with a uniform Lipschitz constant for $x\in O$ and $s$ large enough. Since
\[
u^+_{\tau_s}(x)- u^+_{\tau_s}(y)=d_s(y)-d_s(x),
\]
we only need to show that $d_s$ are equi-Lipschitz coninuous functions. To prove this result, the following lemma is useful.
\begin{lemma}[\protect{\cite[Lemma 14.20]{Beem 1996}}]\label{lem8}
Let U be an open convex domain in $\mathbb{R}^n$ and $f:U\rightarrow \mathbb{R}$ a continuous function. Suppose that for each $p$ in $U$ there is a smooth lower support function $f_p$ defined in a neighborhood of $p$ such that $\|d(f_p)p\|_h<L$. Then $f$ is Lipschitz continuous with Lipschitz constant $L$, i.e., for all $x, y\in U$ we have
\begin{equation}\label{E9}
           |f(x)-f(y)|\leq L|x-y|_h.
\end{equation}
\end{lemma}


 Note $-u^+_{p, s}(x)$ defined in Definition \ref{Def2} is a lower support function to $-u^+_{\tau_s}$. To apply  Lemma \ref{lem8}, we need to establish an estimate for the Lipschitz constant of  lower support functions $-u^+_{p, s}(x)$. By the method in \cite[Lemma 3.3]{EschenburgJDG1988}, $u^+_{p, s}(x)$ are equi-Lipschitz on $U$. Let $g_{i,j}$ be the $i, j$-th component of metric $g$ in the local chart $U$, then the Lipschitz constant $L$ of $-u^+_{p, s}(x)$ can be written as $L=GC$, where $G=\sup\{g_{i,j}(x):x\in U, 1\leq i, j\leq n\}$ and $C$ is as in Lemma \ref{lem10}. Then by the Arzela-Ascoli theorem and $u^+_{\tau_s}(x_0)=0$ (for example, see \cite[Lemma 4.4]{Flores Mams 2013}), there exists a subsequence $s_n$ and a locally Lipschitz function $u^+(x)$ such that
\[
\lim\limits_{s_n \rightarrow +\infty} u^+_{\tau_{s_n}}(x)=u^+(x) ~~\text{for any}~~ x\in M.
\]

So far, we have obtained the local Lipschitzness of $u^+$. In the following we will show that $u^+$ is a locally semiconcave function. Local semiconvexity (semiconcavity) is defined as follows.
\begin{definition}[\protect{\cite[Definition 2.6]{CuiJMP2014}}]\label{Def5}
Let $O$ be an open subset of $M$. A function $\psi: O\rightarrow \mathbb{R}$ is said to be semiconcave if there exists a $c>0$ such that for any constant-speed geodesic (with respect to $h$) path $\gamma(t)$, $t\in[0, 1]$, $\gamma(t)\in O$,
\begin{equation}\label{E8}
         \psi(\gamma(t))\geq (1-t)\psi(\gamma(0))+t\psi(\gamma(1))+c\frac{t(1-t)}{2}d^2_h(\psi(\gamma(0)),\psi(\gamma(1))).
\end{equation}

A function $\psi: M\rightarrow \mathbb{R}$ is said to be locally semiconcave if for each $p\in M$ there is a neighborhood $O$ of $p$ in $M$ such that (\ref{E8}) holds true as soon as $\gamma(t) \in O$ ($t\in [0, 1]$); or equivalently if (\ref{E8}) holds for some fixed positive number $c$ as long as $\gamma$ stays in a compact subset $K$ of $O$.

Similar definitions for semiconvexity and local semiconvexity are obtained in an obvious way by reversing the sign of the inequality in (\ref{E8}).
\end{definition}

To show $u^+(x)$ is locally semiconcave, we need the following lemma. Although the lemma is stated in $\mathbb{R}^n$, it is still valid for our case, since the property we want to study is local.

\begin{lemma}[\protect{\cite[Lemma 3.2]{AnderssonCQG1998}}]\label{lem9}
Let $U$ be an open convex domain in $\mathbb{R}^n$ and $f:U\rightarrow \mathbb{R}$ a continuous function. Assume for some constant $c$ and for all $p\in U$ that $f$ has a smooth upper support function $f_p$ at $p$, i.e. $f_p(x)\geq f(x)$ for all $x$ near $p$ with equality holding when $x=p$, such that $D^2f_p(x)\leq c I$ near $p$. Then $f-\frac{c}{2}\|x\|^2_E$ is concave in $U$, thus $f$ is semiconcave and twice differentiable almost everywhere in $U$. In this lemma, $\|\cdot\|_E$ denotes the Euclidean norm on $R^n$.
\end{lemma}

\textit{Proof of Proposition \ref{the1}}. Without loss of generality, we can assume that $\bar{O}$ in Definition \ref{Def2} is a compact subset of $M$. For any $x\in O$, let $\gamma_{x, s}$ be a maximal geodesic between $x$ to $p_{1,s}$. Then by Lemma \ref{lem9} and the definition of $u^+_{p, s}(x)$, we only need to use comparison theory to give an estimation of the Hessian of the distance function $d(x, p_{1_s})$$(x\in O)$ in terms of upper and lower bounds of the timelike sectional curvatures of  2-planes containing $\dot{\gamma}_{x, s}(0)$ and the length of $\gamma_{x, s}$, where the Hessian of $d(x, p_{1,s})$ is defined in terms of the Levi-Civita connection with respect to $g$. Since $\gamma$ (see Definition \ref{Def2}) is a maximal segment between $x$ to $q_s$, the interior of $\gamma$ is free of cut points. By \cite[Proposition 3.2]{ERDGTG2003}, \cite[Propositions 9.7, 9.29]{Beem 1996}, we can choose $O$ such that the distance function $d(x, p_{1,s})$ is a smooth function with respect to $x\in O$.  When $p_{1,s}$ is close to $x$, there is a compact subset $K_1\subseteq TM$ that contains all the tangent vectors $\dot{\gamma}_{x, s}|_{ x\in O}$. This, together with Lemma \ref{lem24}, implies that both the timelike sectional curvatures of  2-planes containing $\dot{\gamma}_{x, s}(0)$ and the length of $\gamma_{x, s}$ are bounded from above. By the method in \cite[Theorem 5.10]{JinGD2018}, there exists a $c(O)>0$ depending only on $O$ such that $D^2d(x,p_{1,s})\geq -c(O)I$ for $x\in O$. Then, we can finally conclude that $u^+(x)$ is locally semiconcave since the constant $c(O)$ is independent on $s_n$. \qed

\subsection{ $u^+$ is global viscosity solution to equation \eqref{E1}}

In this subsection, we will prove the function $u^+$ defined in Definition \ref{Def1} is a viscosity solution of Lorentzian eikonal equation (\ref{E1}). The proof is slightly modified from the one in \cite[Proposition 2.1]{CuiJMP2014}.

\begin{lemma}\label{lem2}
 If $u^+$ is differentiable at $p$, then $g(\nabla u^+(p), \nabla u^+(p))=-1$ and $\nabla u^+(p)$ is a past-directed timelike vector.
\end{lemma}
\begin{proof}
  Assume that $u^+$ is differentiable at $p$. By the definition of $u^+$, there exists a subsequence $s_n$ of $s\rightarrow \infty$ such that
\[
u^+=\lim\limits_{s_n\rightarrow \infty}u^+_{\tau_{s_n}}.
\]
 Choosing any smooth future-directed causal curve $\gamma: [0, T) \rightarrow M (T>0)$ with $\gamma(0)=p$, we have
\begin{align*}
u^+(\gamma(t))-u^+(p)&=u^+(\gamma(t))-u^+(\gamma(0))\\
&=\lim\limits_{s_n\rightarrow \infty} u^+_{\tau_{s_n}}(\gamma(t))-\lim\limits_{s_n\rightarrow \infty} u^+_{\tau_{s_n}}(\gamma(0))\\
&=\lim\limits_{s_n\rightarrow \infty}(d(\gamma(0), \tau_{s_n})-d(\gamma(t), \tau_{s_n}))\\
&\geq d(\gamma(0), \gamma(t))\\
&\geq \int_0^t \sqrt{-g(\dot{\gamma}(l), \dot{\gamma}(l))} dl
\end{align*}
for every $t\in [0, T)$. Dividing by $t$ on both sides, we get
\begin{equation}\label{E2}
                 \frac{u^+(\gamma(t))-u^+(p)}{t}\geq \frac{1}{t}\int_0^t \sqrt{-g(\dot{\gamma}(l), \dot{\gamma}(l))} dl.
\end{equation}
Letting $t\rightarrow 0_+$, the differentiability of $u^+$ at $p$, together with inequality (\ref{E2}), impiles
\[
d u^+(p)(\dot{\gamma}(0))\geq \sqrt{-g(\dot{\gamma}(0), \dot{\gamma}(0))}.
\]
In other words
\begin{equation}\label{E3}
                 g(\nabla u^+(p), \dot{\gamma}(0))\geq \sqrt{-g(\dot{\gamma}(0), \dot{\gamma}(0))}
\end{equation}
for any future-directed causal vector $\dot{\gamma}(0)$. To continue the proof of Lemma \ref{lem2}, we need the following two claims.

\textit{Claim 1}. $\nabla u^+(p)$ is a past-directed timelike vector.

First of all, by (\ref{E3}) we have $\nabla u^+(p)\neq 0$.  Suppose $\nabla u^+(p)$ is spacelike, then there exists a smooth future-directed timelike curve $\bar{\gamma}: [0, \epsilon) \rightarrow M$ such that $\bar{\gamma}(0)=p$ and $g(\dot{\bar{\gamma}}(0), \nabla u^+(p))=0$, this means that $\nabla u^+(p)$ and $\dot{\bar{\gamma}}(0)$ are orthogonal with respect to the Lorentzian metric $g$. On the other hand, form (\ref{E3}), we get $g(\nabla u^+(p), \dot{\bar{\gamma}}(0))>0$, this contradicts the orthogonal hypothesis. Thus, $\nabla u^+(p)$ is a causal vector. Furthermore, $g(\nabla u^+(p), V)\geq 0$ for any future-directed causal vector $V\in T_p M$ implies $\nabla u^+(p)$ is past-directed. To show $\nabla u^+(p)$ is indeed a timelike vector, we need the following lemma

\begin{lemma}\label{lem4}
  Suppose $V \in T_p M$ is a nonzero past-directed non-spacelike vector and $g(V, W)\geq \sqrt{-g(W, W)}$ for any future-directed causal vector $W$, then $V$ is not lightlike.
\end{lemma}
The proof of Lemma \ref{lem4} is omitted and we refer to \cite[Lemma 2.5]{CuiJMP2014}.

\textit{Claim 2}. $|\nabla u^+(p)|=1$.

  First we choose a future-directed smooth causal curve $\delta: [0, \epsilon)\rightarrow M$ with $\delta(0)=p$, $\dot{\delta}(0)=-\nabla u^+(p)$. Then by inequality (\ref{E3}),
\begin{align*}
|\nabla u^+(p)|^2=-g(\nabla u^+(p), \nabla u^+(p))=g(\nabla u^+(p), \dot{\delta}(0))\geq\sqrt{-g(\dot{\delta}(0), \dot{\delta}(0))}=|\nabla u^+(p)|.
\end{align*}
Hence, either $|\nabla u^+(p)|\geq 1$ or $|\nabla u^+(p)|=0$. By \textit{Claim 1}, we obtain
\begin{equation}\label{E6}
  |\nabla u^+(p)|\geq 1            .
\end{equation}
By Remark \ref{remark1}, there exists a maximal geodesic segment between $p$ and $\tau_s$ denoted by $\gamma_{p,s}$, i.e., $\gamma_{p,s}$ is a future-directed timelike (unit-speed) geodesic $\gamma_{p,s}:[0, d(p, \tau_s)]\rightarrow M$ with $\gamma_{p,s}(0)=p$, such that $d(p, \tau_s)= L(\gamma_{p,s})$ and for $0\leq t_1 <t_2 \leq d(p, \tau_s)$,
\begin{align*}
d(\gamma_{p,s}(t_1), \gamma_{p,s}(t_2))=t_2 -t_1.
\end{align*}
Then we get
\begin{align*}
\frac{u^+(\gamma_{p,s}(t))-u^+(\gamma_{p,s}(0))}{t-0}=\frac{\lim\limits_{s_n\rightarrow\infty}\big[d(\gamma_{p,s_n}(0), \tau_{s_n})-d(\gamma_{p,s_n}(t), \tau_{s_n})\big]}{t-0}=1
\end{align*}
for every $t\in (0, d(p, \tau_s))$. Letting $t\rightarrow 0^+$, by the reverse Cauchy-Schwarz inequality for causal vectors (\cite[2.2.1]{SachsWuBAMS1977}) and the fact that $\gamma_{p,s}$ is a timelike (unit-speed) geodesic,  we have
\begin{equation}\label{E7}
  1=d u^+(p)(\dot{\gamma}_{p,s}(0))=g(\nabla u^+(p), \dot{\gamma}_{p,s}(0))\geq |\nabla u^+(p)| |\dot{\gamma}_{p,s}(0)| = |\nabla u^+(p)|       .
\end{equation}
 By inequalities (\ref{E6}) and (\ref{E7}), the proof of \textit{Claim 2} is complete.

So far, the proof of Lemma \ref{lem2} is finished. \qed
\end{proof}

   For a set $B$ in a vector space, the convex hull of $B$, $co B$, is the smallest convex set containing $B$. 

\begin{lemma}\label{lem6}
  If $\psi$ is a locally semiconvex (resp., semiconcave) function on manifold $M$, then it is locally Lipschitz (under any reasonable metric), and $\nabla^- \psi(p)$ (resp., $\nabla^+ \psi(p)$) is nonempty for any $p\in M$. In this case, $\nabla^- \psi(p) (resp., \nabla^+ \psi(p))=co\nabla^*\psi(p) \subset T_pM$.
\end{lemma}
 For a proof of Lemma \ref{lem6}, we refer to \cite[Theorem 3.3.6]{CannarsaPNDE2004}, where the limiting gradient is called reachable gradient.

\begin{lemma}\label{lem7}
  Let $(M, g)$ be a globally hyperbolic space-time, then the function $u^+$ defined in Definition \ref{Def1} is a viscosity solution to Lorentzian eikonal equation (\ref{E1}) .
\end{lemma}

\begin{proof}
By the definition of viscosity solution, we need to show $u^+$ is subsolution and supersolution of equation (\ref{E1}). Firstly, for any $V\in \nabla^+ u^+(p)$,
by Lemma \ref{lem6},
\begin{align*}
V\in \nabla^+ u^+(p)=co \nabla^*u^+(p).
\end{align*}
By Lemma \ref{lem2} and the convexity of the set $\{W\in T_p M | ~W ~\text{is past-directed timelike vector and}~ g(W, W)\leq -1 \}$, we can finally conclude that
\[
g(V,V)\leq -1.
\]
It means that $u^+$ is a subsolution of the equation (\ref{E1}).

On the other hand, since $u^+$ is locally semiconcave, $u^+$ is differentiable at $p\in M$ whenever $\nabla^- u^+(p)\neq \emptyset$. Thus, Lemma \ref{lem2} implies that $u^+$ is a supersolution.

Therefore $u^+$ is indeed a viscosity solution of the equation (\ref{E1}) and this completes the proof of Theorem \ref{the10}.\qed
\end{proof}

\begin{remark} 	
Using the same method, we can finally conclude that $u^-$ defined in Definition \ref{Def1} is also a global viscosity solution to Lorentzian eikonal equation (\ref{E1}).
%
\end{remark}

\section{The level sets of viscosity solution of Lorentzian eikonal equation}\label{sec:level sets}
In this section, for any $u\in \mathcal{S}(M)$ we shall study the level sets of $u$ when the time orientation of $u$ is consistent.  Firstly, we give an  equivalent characterization of time orientation of $u$. Secondly, we show that $u_s$ is a partial Cauchy surface for each $s\in \text{Image}(u)$. Finally, we give an example, which shows that in general the level set $u_s$ is not a Cauchy surface. Recall that the time orientation of $u$ is consistent means that for any $x, y\in M$, $V\in \nabla^* u(x)$ and $W\in \nabla^* u(y)$, $V$ and $W$ have the same time orientation.  Without loss of generality, we always assume that the time orientation of $u$ is past-directed and $s\in \text{Image}(u)$ in this section.

\subsection{Equivalent characterization of time orientation of viscosity solutions}

\begin{lemma}\label{lem11}
  For $u\in \mathcal{S}(M)$, the time orientation of $u$ at any point is past-directed if and only if $u$ is increasing monotonically along any inextendible future-directed causal curve.
\end{lemma}

\begin{proof}
Assume that the time orientation of $u$ at any point is past-directed and let $\zeta$ be an inextendible future-directed causal curve. We firstly consider the case that $u$ is differentiable almost everywhere on $\zeta$. For any $a<b\in \text{dom}(\zeta)$, we have
\begin{align}\label{E45}
u(\zeta(b))-u(\zeta(a))&=\int_a^b d_{\zeta(s)}u(\dot{\zeta}(s))ds   \nonumber\\
                     &= \int_a^b g(\nabla u, \dot{\zeta})|_{\zeta(s)} ds.
\end{align}
Since $u$ is locally Lipshcitz with respect to $h$, for any $x\in M$, there exists a compact neighborhood $U_x$ such that
\[
\inf\limits_{z\in U_x, V\in \partial\mathcal{C}_z^{-1}}\left\{\left|\frac{\nabla u(z)}{|\nabla u(z)|_h}-V \right|_h\right\}>0,
\]
where $\partial\mathcal{C}_z^{-1}:=\{ ~V\in T_z M|~ V ~\text{is a past-directed lightlike vector and}~ |V|_h=1\}$. Thus, we have
\[
\inf\limits_{z\in U_x, W\in \mathcal{C}_z}g(\nabla u(z), W)>0,
\]
where $\mathcal{C}_z: = \{ ~V\in T_z M|~ V ~\text{is a future-directed causal vector and}~ |V|_h=1\}$. Without loss of generality, we assume that $\zeta(a)$ and $\zeta(b)$ are both in $U_x$. We parametrize $\zeta|_{[a, b]}$ with respect to $h$-arc length and use $\tilde{\zeta}:[\tilde{a}, \tilde{b}]\rightarrow M$ to denote this reparametrized curve. By equality \eqref{E45}, we have
\begin{align}\label{E46}
u(\zeta(b))-u(\zeta(a))&=\int_{\tilde{a}}^{\tilde{b}} g(\nabla u, \dot{\tilde{\zeta}})|_{\tilde{\zeta}(s)} ds \nonumber\\
                     &\geq d_h(\zeta(a), \zeta(b))\inf\limits_{z\in U_x, W\in \mathcal{C}_z}g(\nabla u(z), W).
\end{align}

In the following, we will study the case that $u$ is not differentiable almost everywhere on $\zeta$. For any $a<b\in \text{dom}(\zeta)$, we choose an open subset $U\subseteq M$ such that $\zeta|_{[a,b]} \subseteq U$. Since the time orientation of $u$ is past-directed, we can find a sequence of future-directed causal curves $\zeta_i$ satisfying
\begin{itemize}
  \item $\zeta_i\subseteq U$ for each i.
  \item $\zeta_i\rightarrow \zeta$ in the $C^0$-topology.
  \item $u$ is differentiable almost everywhere on $\zeta_i$.
\end{itemize}
We parametrize these curves $\zeta_i|_{[a_i, b_i]}$ with respect to $h$-arc length and use $\tilde{\zeta}_i:[\tilde{a_i}, \tilde{b_i}]\rightarrow M$ to denote these curves. Using the same argument above,
\begin{align*}
u(\zeta(b))-u(\zeta(a))&=\lim\limits_{i\rightarrow\infty}\int_{\tilde{a}_i}^{\tilde{b}_i} d_{\tilde{\zeta}_i(s)}u(\dot{\tilde{\zeta}}_i(s))ds\\
                     &\geq \lim\limits_{i\rightarrow\infty}d_h(\tilde{\zeta}(\tilde{a}_i), \tilde{\zeta}(\tilde{b}_i))\inf\limits_{z\in U_x, W\in \mathcal{C}_z}g(\nabla u(z), W)\\
                     &= d_h(\zeta(a), \zeta(b))\inf\limits_{z\in U_x, W\in \mathcal{C}_z}g(\nabla u(z), W) \\
                     &>0.
\end{align*}
By the arbitrariness of $a$, $b$ and $\zeta$, $u$ is increasing monotonically along any inextendible future-directed causal curve.
To prove the other direction of the lemma, we propose the following claim.

\textit{Claim:} Let $x$ be a differentiable point of $u$, then there exists a neighborhood $U$ of $x$ such that for any differentiable point $y\in U$, the time orientation of $u$ at $y$ is consistent with the time orientation of $u$ at $x$.

\textit{Proof of  Claim.}  Actually, if the claim is not true, we can get a sequence $\{x_i\}$ with $x_i\rightarrow x$ and the time orientation of $u$ at $x_i$ is different from the one at $x$. Then there exists a timelike vector in $\nabla^* u(x)$, which is opposite to the time orientation of $u$ at $x$. This contradicts the hypothesis that $x$ is a differentiable point of $u$. Thus the \textit{Claim} is proved. 


If there exists a future-directed timelike vector $V\in \nabla^* u(x)$, by the definition of $\nabla^*$, we can get a sequence of differentiable points $x_i$ such that $\nabla u(x_i)$ is future-directed timelike for each $x_i$. By the \textit{Claim} and the analysis above, we can easily find a future-directed causal curve $\gamma$ in a neighborhood of $x_i$ such that $u$ is decreasing monotonically on $\gamma$. One can extend $\gamma$ to an inextendible causal curve. It is a contradiction to the fact that $u$ is increasing monotonically along any inextendible future-directed causal curve. Up to now, the proof of Lemma \ref{lem11} is complete. \qed
\end{proof}

\subsection{The level sets of $u$ are partial Cauchy surfaces}\label{subsection4.2}

\begin{figure}\label{p2}
\begin{center}
\begin{tikzpicture} [scale=1.2]
\draw[->] (-3,0) -- (3,0) node[right]{$x$};
\draw[->] (0,-3) -- (0,3) node[right]{$y$};
\filldraw (0,0) circle (0.5pt) node[anchor=north east] {$O$};
\filldraw (-2,0) circle (1pt) node[anchor=north east] {$\scriptstyle (-1,0)$};
\filldraw (2,0) circle (1pt) node[anchor=north west] {$\scriptstyle (1,0)$};
\filldraw [blue](0,-2) circle (1pt) node[anchor=north west] {$\scriptstyle (0,-1)$};
\filldraw [blue](0,2) circle (1pt) node[anchor=south west] {$\scriptstyle (0,1)$};
\draw (2,0) .. controls (1,1) .. (0,2);
\draw (-2,0) .. controls (-1,1) .. (0,2);
\draw (-2,0) .. controls (-1,-1) .. (0,-2);
\draw (0,-2) .. controls (1,-1) .. (2,0);
\draw [blue] (0,-2) .. controls (0.5,0)  .. (0,2) ;
\draw [blue](0.5,0.25) node[above] {$\Gamma$};
\draw [red, thick](1.2,-0.8) .. controls (1.2,0) .. (1.2,0.8);
\draw [red] (1.2,0.2) node[right] {$u_a$};
\filldraw [red](1.2,0) circle (1pt) node[anchor=north east] {$\scriptstyle (a,0)$};
\end{tikzpicture}
\caption{A space-time admits a viscosity solution $u(x,y)=x$. Obviously, any Cauchy surface $\Gamma$ (the blue curve) must connect $(0, -1)$ and $(0, 1)$. The level set $u_a$ (the red segment) is not a Cauchy surface except $a=0$.}
\end{center}
\end{figure}
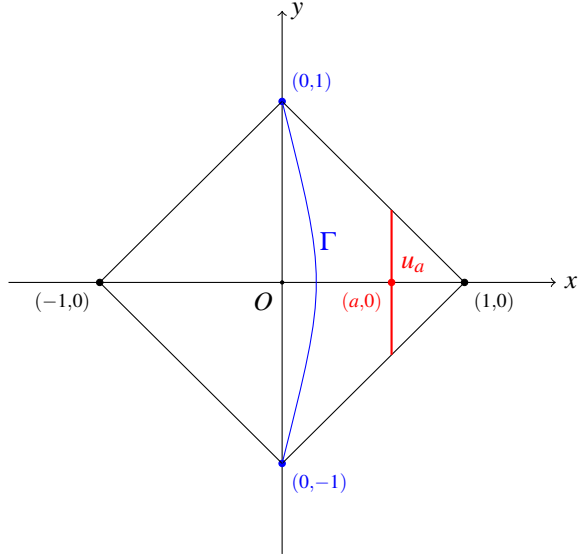

\begin{lemma}\label{lem16}
 Fix $s\in \text{Image}(u)$, $u_s$ is a closed acausal set.
\end{lemma}
\begin{proof}
Indeed, $u_s$ is a closed subset by the continuity of $u$. We only need to show that $u_s$ is an acausal set. For $x_1, x_2 \in u_s$, if there exists a future-directed causal curve $\theta$ connecting $x_1$ and $x_2$, without loss of generality, we can assume that $\theta(0)=x_1$, $\theta(1)=x_2$ and $u$ is differentiable almost everywhere on $\theta$. By Lemma \ref{lem11}, we have
\begin{align*}
u(x_2)>u(x_1)
\end{align*}
This contradicts the fact $x_1, x_2 \in u_s$.\qed
\end{proof}

\begin{lemma}\label{lem17}
 For each $s\in \text{Image}(u)$, $\text{edge}(u_s)=\emptyset$ .
\end{lemma}
\begin{proof}
On the contrary, for any arbitrarily fixed $s\in \text{Image}(u)$ and any $p\in \text{edge}(u_s)$, by the definition of $\text{edge}(u_s)$, there exists a future-directed timelike curve $\beta:[a,b]\rightarrow M$ with $\beta(a)\ll p \ll \beta(b)$ and $\beta(t)\cap u_s=\emptyset$ for any $t\in [a,b]$. By Lemma \ref{lem11}, $u$ is strictly monotonically increasing along any future-directed causal curve. Thus we must have $u(\beta(0))<u(p)<u(\beta(t))$, this inequality  together with the continuity of $u$ implies that there exists a $t_0\in [a,b]$ such that $u(\beta(t_0))=u(p)$. This is impossible, since $\beta$ is chosen such that $\beta(t)\cap u_s=\emptyset$ for any $t\in[a, b]$. \qed
\end{proof}

Up to now, by Definition \ref{Def10}, $u_s$ is a partial Cauchy surface for each $s\in \text{Image}(u)$. To conclude this section we provide an example, which shows that in general the level sets $u_s$ are not Cauchy surfaces.

\textit{Example}.
In the 2-dimensional Minkowski space-time $(\mathbb{R}^2, dy^2-dx^2)$, let $M=I^+((-1,0))\cap I^-((1,0))$ be the induced space-time. Clearly, $M$ is globally hyperbolic. It is easy to see that $u(x,y)=x$ is a globally defined viscosity solution of equation (\ref{E1}) on $M$, but level sets $u_s$ are only partial Cauchy surfaces except $u_0$. See Figure 1.


\section{A variational representation of $u$}\label{sec:uuu}

In this section, we will give a variational representation of $u$ when the time orientation of $u$ at any point is past-directed. First of all, we propose a variational problem and prove the existence of minimizers. In the second place, we prove equality \eqref{25} and the existence of forward calibrated curve of $u$. In the end, we show that $u$ is a locally semiconcave function. These results in company complete the proof of the Theorem \ref{the7}. Our proof is motivated by the weak KAM theory for positively definite Lagrangian systems \cite{FathiWKAM}. To go into the details, we first fix a notation. In this section, $\Omega_{x}^t$ denotes the set of all the past-directed piecewisely smooth timelike curve $\gamma: [0, t]\rightarrow M$  with $\gamma(t)=x$.

\subsection{The existence of minimizer to the variational problem}
Define
\begin{equation}\label{E26}
\tilde{u}(t,x):=\inf\limits_{\gamma\in \Omega_{x}^t}\left[u(\gamma(0))-\int_0^t\sqrt{-g(\dot{\gamma}(s), \dot{\gamma}(s))}ds\right].
\end{equation}
Firstly, we shall prove the existence of minimizers in this variational problem.
\begin{lemma}\label{lem18}
For any $x\in M$, there exists a $t_0>0$ depending on $x$ and a past-directed timelike curve $\gamma_x\in \Omega_{x}^t$ such that
\begin{equation}\label{E27}
\tilde{u}(t,\gamma_x(t))=u(\gamma_x(0))-\int_0^t\sqrt{-g(\dot{\gamma}_x(s), \dot{\gamma}_x(s))}ds
\end{equation}
for any $t\leq t_0$.
\end{lemma}
\begin{proof}
For any $x\in M$ and $\gamma\in \Omega_x^t$, $L(\gamma)\leq u(\gamma(0))-u(x)$, and consequently $\tilde{u}(t,x)\geq u(x)$. Since the maximal existence time of ODE is lower semi-continuous with respect to initial conditions, we can find a uniformly $t>0$ such that any future-directed casual curve $\gamma$ with $\gamma(0)=x$ is well-defined on $[0, t]$. Moreover, by Lemma \ref{lem11} and inquality \eqref{E46}, there exists a positive constant $\delta$ which depends only on $d_h(\gamma(0), \gamma(t))$ and the local Lipschitz constant of $u$ such that $u(\gamma(t))-u(\gamma(0))\geq \delta>0$. Therefore, there exists an $s\in \text{Image}(u)$ such that $x\in D^-(u_s)\setminus u_s$. For any $z\in J^+_{u_s}(x)$, we set $t_{x, z}=\frac{1}{2}d(x, z)$ and $t_0=\max\limits_{z\in J^+_{u_s}(x)}t_{x, z}$.  Recall that by Lemmas \ref{lem16} \ref{lem17}, for each $s\in \text{Image}(u)$, $u_s$ is an acausal set and $\text{edge}(u_s)=\emptyset$. Then by \cite[Corollary 14.26]{Oneill1983}, $u_s$ is a closed topological hypersurface. Thanks to \cite[Lemma 14.43]{Oneill1983}, we know that $D(u_s)$ is open. By Remark \ref{remark1}, \cite[Proposition 14.31]{Oneill1983} and \cite[Theorem 14.38]{Oneill1983} we obtain that $J^+_{u_s}(x)$ is a compact, connected subset of $M$. Then, for the set $\{p\in I^+(x)| d(x,p)=t_0\}$, there exists a point $y\in u_s$ such that $d(x, y)= t_0$. Since $D(u_s)$ is open, Lemma \ref{lem25} guarantees that for any $x\in D^-(u_s)\setminus u_s$, the set $J^+(x)\cap D^-(u_s)$ is compact. Since the Lorentzian length functional $L(\gamma)$ does not rely on parameterization of $\gamma$, for the minimizer $\gamma$ of  \eqref{E26}, we can always assume that $\gamma$ is a unit speed maximal geodesic, then it is easy to check that
\begin{align}\label{E39}
\tilde{u}(t_0, x)=\inf\limits_{x\leq y, d(x,y)=t_0 }u(y)-t_0.
\end{align}
Then for any $t\leq t_0$, the  minimizer of  \eqref{E39} must be contained in $\{y| d(x,y)= t_0\}\cap J^-[u_s]$, which is a closed subset of $J^+(x)\cap D^-(u_s)$. Because of the upper semi-contimuity of $L(\cdot)$ and the existence of maximal geodesic between any two points $x, y$ with $x\leq y$ in $M$, we can finally find $\gamma_x$ satisfies equation (\ref{E27}).  \qed
\end{proof}

In the rest of this section, for any $x\in M$, we still use $\gamma_x^t$ to denote one of the minimizers in equality \eqref{E26}.


\subsection{ Variational representation and forward calibrated curve }\label{subs5.2}
 In this subsection, we shall show that the viscosity solution $u$ admits a local variational representation. Such variational representations can help us obtain an inextendible forward calibrated curve of $u$ at any $x\in M$. Based on Lemma \ref{lem18}, we have the following lemma.
\begin{lemma}\label{lem22}
 For any $x\in M$, there exists a $t_0>0$ depending on $x$ such that
\begin{align}\label{E40}
u(x)=\tilde{u}(t,x)=\inf\limits_{x\leq y, d(x,y)=t}u(y)-t
\end{align}
for any $0<t\leq t_0$.
\end{lemma}
\begin{proof}
By Lemma \ref{lem18}, there exists a $t_0>0$ depending on $x$ such that the second equality holds for any $t\in (0, t_0]$. Then, we only need to show that $u(x)=\tilde{u}(t,x)$ for any $t\in [0, t_0]$. For $0<t_1<t_2<t_0$, by Lemma \ref{lem18}, we have two past-directed timelike curves $\gamma_x^{t_1}:[0, t_1]\rightarrow M$ and $\gamma_x^{t_2}:[0, t_2]\rightarrow M$. Moreover, $x=\gamma_x^{t_1}(t_1)=\gamma_x^{t_2}(t_2)$. We define a past-directed timelike curve $\gamma_x^{t_1,t_2}:[0, t_2]\rightarrow M$ with $\gamma_x^{t_1,t_2}(s)=\gamma_x^{t_1}(\frac{t_1}{t_2}s)$ for any $s\in [0, t_2]$. Since the Lorentzian length functional $L(\gamma)$ does not rely on parameterization of $\gamma$, it can be concluded that
\begin{align*}
\tilde{u}(t_2,x)-\tilde{u}(t_1,x)\leq u(\gamma_x^{t_1, t_2}(0))-L(\gamma_x^{t_1, t_2})-u(\gamma_x^{t_1}(0))+L(\gamma_x^{t_1})=0.
\end{align*}
Similarly, it can be inferred that
\begin{align*}
\tilde{u}(t_2,x)-\tilde{u}(t_1,x)\geq 0.
\end{align*}
Thus, $\tilde{u}(t,x)$ is a constant with respect to $t$. Letting $t\rightarrow 0$, we can finally conclude that
\[
\tilde{u}(t,x)=u(x)
\]
\qed
\end{proof}

Up to now, we have shown that $u(x)=\inf\limits_{x\leq y, d(x,y)=t }u(y)-t$ for any $t\in [0, t_0]$. Without loss of generality, in the rest of this paper , we always assume that $\gamma_x^{t_0}$ is a unit speed maximal geodesic segment. For any $s\in [0, t_0]$, we define
\[
\breve{\gamma}^{t_0}_x(s)=\gamma^{t_0}_x(t_0-s).
\]
Then we have a future-directed unit speed timelike curve $\breve{\gamma}^{t_0}_x: [0 , t_0]\rightarrow M$ satisfying
\begin{align}\label{E41}
u(\breve{\gamma}_x^t(s_2))-u(\breve{\gamma}_x^t(s_1))=s_2-s_1,
\end{align}
for any $0\leq s_1 \leq s_2\leq t_0$. For $\breve{\gamma}_x(t_0)$, using Lemma \ref{lem22} again, we can extend $\breve{\gamma}_x^{t_0}:[0, t_0]\rightarrow M$ to be a future-directed timelike curve (denoted by $\breve{\gamma}_x^{t_0+t_1}$) $\breve{\gamma}_x^{t_0+t_1}:[0, t_0+t_1]\rightarrow M$. Obviously, $\breve{\gamma}_x^{t_0+t_1}$ satisfies equality \eqref{E41} for any $0\leq s_1 \leq s_2\leq t_0+t_1$. By the reverse triangle inequality of $d$, we have
\begin{align}\label{E50}
t_0+t_1\leq d(x, \breve{\gamma}_x^{t_0+t_1}(t_0+t_1)).
\end{align}
Inductively, one can easily get a future-directed inextendible timelike curve, denoted by $\gamma_x$. Actually, if $\gamma_x$ has an endpoint $p\in M$, by the reverse triangle inequality and \eqref{E50}, 
\[
\sum \limits_{i} t_i \leq d(x, p).
\]
For the point $p$, by Lemma \ref{lem22}, there exists a $t_p>0$, such that $u(p)=\inf\limits_{x\leq y, d(x,y)=t}u(y)-t$ for any $0<t\leq t_p$. Then $\gamma_x$ can be extended again.

In the rest of this paper, we always use $\gamma_x$ and $T$ to denote the future-directed inextendible timelike curve of $u$ at $x$ and the maximal existence time of $\gamma_x$.
Since $\gamma_x$ satisfies equality \eqref{E41} for any $0<s_1<s_2<T$, $\gamma_x$ is an inextendible forward calibrated curve of $u$ at $x$. In the following, we shall show that $\gamma_x$ is a ray. Fix arbitrarily $t_1, t_2\in[0, T)$, for any timelike curve $\eta$  connecting $\gamma_x(t_1)$ and $\gamma_x(t_2)$, we firsly assume that $u$ is differentiable almost everywhere on $\eta$. Then by the reverse Cauchy-Schwarz inequality, we have
\[
L(\gamma_x)|_{(t_1,  t_2)}=u(\gamma_x(t_2))-u(\gamma_x(t_1))\geq\int_{t_1}^{t_2} g(\nabla u(\eta), \dot{\eta}) ds \geq L(\eta).
\]
If $u$ is not differentiable almost everywhere on $\eta$, using the method in Lemma \ref{lem11}, we still have the above inquality.
This means that the inextendible timelike curve $\gamma_x$ is maximal on each segment. In a word, $\gamma_x$ is an inextendible forward calibrated ray.

\subsection{Local semiconcavity of $u$ }
In this subsection, we will show that $u$ is locally semiconcave when the time orientation of $u$ is consistent. Actually, with the help of variational representation of $u$, one can directly obtain the semiconcavity of $u$ in a local chart. The following proposition is the main result of this subsection.

\begin{proposition}\label{prop4}
Let $u\in \mathcal{S}(M)$, suppose that the time orientation of $u$ at any point is past-directed, then for any $x\in M$, there exists an $s\in \text{Image}(u)$ and a neighborhood $U$ of $x$ such that up to a constant,
\[
u(x)=-d(x, u_s)
\]
for any $x\in U$.
\end{proposition}
Proposition \ref{prop4} can be seen as a space-time counterpart of the results obtained in \cite{CuiIJM2016}, where the author showed that any viscosity solution is somehow a kind of distance like function in the Riemannian case.  Distance like functions, introduced by Gromov, play an important role in studying the geometry and topology of some kind of non-compact manifold. For more details, readers can refer \cite{CuiIJM2016, GromovHG}. The proof of Proposition \ref{prop4} is based on the following lemmas.
\begin{lemma}\label{lem21}
For any unit vector $V\in \nabla^* u(x)$, there exists a constant $c>0$ and a unique future-directed timelike geodesic $\gamma_{x,V}:[0, c]\rightarrow M$, satisfies
\[
\gamma_{x,V}(0)=x, ~~~u(\gamma_{x,V}(t))=u(\gamma_{x,V}(0))+t ~~\mbox{for every}~~t\in [0, c]
\]
and
\[
\dot{\gamma} _{x,V}(0)=-V.
\]
\end{lemma}
\begin{proof}
For any $V\in\nabla^* u(x)$, by the definition of $\nabla^*u(x)$, there exists a sequence of events $x_i \in M$ with $x_i\rightarrow x$, such that $u$ is differentiable at $x_i$ and $\nabla u(x_i)\rightarrow V$. For each $x_i \in M$, by the last subsection, there exists a future inextendible unit speed timelike geodesic $\gamma_i: [0, t_i)\rightarrow M$, such that
\[
u(\gamma_{i}(t))-u(\gamma_{i}(0))=t
\]
for every $t\in (0, t_i)$. Diving by $t$ and letting $t\rightarrow 0$, we have
\begin{align*}
g(\nabla u(x_i), \dot{\gamma}_{x_i}(0))=1.
\end{align*}
On the other hand, by the reverse Cauchy-Schwarz inequality,
\[
g(\nabla u(x_i), \dot{\gamma}_{x_i}(0))\geq |\nabla u(x_i)||\dot{\gamma}_{x_i}(0)|=1
\]
Thus,
\[
\dot{\gamma}_i(0)=-\nabla u(x_i).
\]

For such a sequence of timelike geodesics $\gamma_i$, by the smooth dependence of solutions of ODE on the initial conditions, there exists a future-directed timelike geodesic $\gamma_\infty: [0, T)\rightarrow M$ with $\gamma_\infty(0)=x$, $\dot{\gamma}_\infty(0)=-V$, where $T$ is the maximal existence time of $\gamma_\infty$, such that $\gamma_i$ converge to $\gamma_\infty$ on any compact subset of $[0, T)$ with respect to $C^1$ topology. Moreover, since the maximal existence time of ODE is lower semi-continuous with respect to initial conditions, we can find a uniform $c>0$ such that $\{\gamma_i|_{[0,c]}\}_i$ and $\gamma_\infty|_{[0,c]}$ are
 well-defined.

 By the continuity of the Lorentzian distance function $d$, the maximality of $\{\gamma_i\}$ and the upper semicontinuity of the length functional with respect to the $C^0$-topology, we have, for each $t\in [0, c]$,
\begin{align*}
d(\gamma_\infty(0),\gamma_\infty(t))&= \lim\limits_{i\rightarrow\infty}d(\gamma_i(0), \gamma_i(t))\\
                      &=\lim\limits_{i\rightarrow\infty}L(\gamma_i|_{[0,t]})\\
                      &\leq L(\gamma_\infty)|_{[0,t]}\\
                      &\leq d(\gamma_\infty(0), \gamma_\infty(t))=d(x,\gamma_\infty(t)).
\end{align*}
Thus $\gamma_\infty$ is a maximal geodesic on $[0, t]$. By the continuity of $u$ and the convergence of $\gamma_i$,
\[
u(\gamma_\infty(t))=\lim\limits_{i\rightarrow\infty} u(\gamma_i(t))=\lim\limits_{i\rightarrow\infty} u(\gamma_i(0))+t=u(x)+t.
\]

Finally, define $\gamma_{x,V}:=\gamma_\infty$, then it is easily seen that $\gamma_{x,V}$ satisfied all the properties we need. So the proof of Lemma \ref{lem21} is complete.\qed
\end{proof}

\begin{lemma}\label{lem13}
Let $u\in \mathcal{S}(M)$ and $s\in \text{Image}(u)$, suppose that the time orientation of $u$ at any point is past-directed, then for any $x\in D^-(u_s)$,  we have $d(x, u_{s})= s- u(x)$.
\end{lemma}
\begin{proof}
For any $x\in D^-(u_s)$, we choose a reachable unit vector $V\in \nabla^* u(x)$. By Lemma \ref{lem21}, there exists a unique future-directed maximal geodesic segment $\beta: [0, c]\rightarrow M$ such that  $\beta(0)=x$, $\dot{\beta}(0)=V$ and $u(\beta(t))-u(x)=t$. By the construction of $t_0$ in Lemma \ref{lem22} and the extensibility of the calibrated curve in subsection \ref{subs5.2}, we can assume that $u(\beta(c))>s$. Thus, there exists a unique $t_1$ such that $u(\beta(t_1))=s$, $t_1=s- u(x)$ and
\[
d(x, \beta(t_1))=\mbox{length}(\beta|[0,t_1])=t_1=s-u(x).
\]
Since $\beta(t_1)\in u_{s}$, we have $d(x, u_{s})\geq s-u(x)$. If $s_1:= d(x, u_{s})> s-u(x)$, applying Fubuni's theorem, we can find an absolutely continuous curve (parameterized by arc-length) $\gamma: [0, s_2]\rightarrow M$ with $\gamma(0)=x\in u_{u(x)}$, $\gamma(s_2)\in u_{s}$ such that $0\leq s_1-s_2 \leq \frac{1}{2}(s_1-s+u(x))$ and $u$ is differentiable almost everywhere along $\gamma$ (with respect to the 1-dimensional Lebesgue measure). By the reverse Cauchy-Schwarz inequality for causal vectors, we have
\begin{align*}
s-u(x)&=u(\gamma(s_2))-u(\gamma(0))\\
       &=\int_0^{s_2} g(\nabla u, \dot{\gamma})|_{\gamma(t)}dt\\
       &\geq \mbox{length}(\gamma|_{[0,s_2]})\\
       &=s_2\\
       &\geq s_1-\frac{1}{2}(s_1-s+u(x))\\
       &> s-u(x).
\end{align*}
This contradiction proves $d(x, u_{s})= s-u(x)$. \qed
\end{proof}

\textit{Proof of Proposition \ref{prop4}.}
For any $x\in M$, by Lemmas \ref{lem18} \ref{lem22} , there exists an $s>0$, which is dependent on $x$, such that $x\in D^-(u_s)$. Then by Lemma \ref{lem13},
\begin{equation}\label{E35}
d(x, u_s)=s-u(x).
\end{equation}
Recall that for each $s\in \text{image}(u_s)$, $D(u_s)$ is an open subset of $M$. Then there exists a neighborhood $U$ of $x$ such that $U\subseteq D^-(u_s)$. By Lemma \ref{lem13}, for any $x\in U$, we have
\begin{equation}\label{E36}
u(x)=-d(x, u_s)+s.
\end{equation}

\textit{Proof of Theorem \ref{the7}.}
The first two conclusions of Theorem \ref{the7} can be obtained from Lemmas \ref{lem16}, \ref{lem17}, \ref{lem22}. For the last conclusion Theorem \ref{the7}, we can construct a support function of $u$ by Proposition \ref{prop4}.  Similar to $u^+(x)$ in Proposition \ref{the1}, the local semiconcavity of $u$ can be obtained by Lemma \ref{lem9} .

\begin{remark}\label{viscosity and smc}
Our results show that when the time orientation of viscosity solution is consistent, viscosity solution is locally semiconcave. Actually, it is not difficult to show that when the time orientation of a function $u\in Lip_{loc}(M)$ is consistent, $u\in \mathcal{S}(M)$ if and only if $u$ is a locally semiconcave function satisfying \eqref{E1} almost everywhere.
\end{remark}

\section{Weak KAM properties of $u$}\label{sec:3}

In this section, we will show that $u$ satisfies some particular properties, so-called weak KAM properties \cite{FathiWKAM} when the time orientation of $u$ is consistent. Similar properties for Busemann functions and regular cosmological time functions have been obtained on space-time in \cite{CuiJMP2014}. In this section, we still assume that the time orientation of $u$ at any point is past-directed.

Recall that if $u$ is differentiable at $x\in M$, then by Lemma \ref{lem21} and the discussions in subsection \ref{subs5.2}, there exists a unique forward calibrated ray $\gamma_x:[0, T)\rightarrow M$ satisfies
\[
\gamma_x(0)=x, u(\gamma_x(t))=u(\gamma_x(0))+t ~~~\mbox{for every}~~~ t\in [0,T)
\]
and
\[
\dot{\gamma_x}(0)=-\nabla u(x).
\]

\begin{proposition}\label{prop2}
If there is a unique future-directed timelike curve $\gamma_x:[0,T)\rightarrow M$, such that
\begin{equation}\label{E17}
  \gamma_x(0)=x, ~~ u(\gamma_x(t))=u(\gamma_x(0))+t  ~~~\mbox{for every}~~~ t\in [0,T),
\end{equation}
then $u$ is differentiable at $x$.
\end{proposition}
\begin{proof}
 By Proposition \ref{prop4}, for any $x\in M$ there exists some $s\in R$ such that $u(x)=-d(x,u_s)$. Moreover, by Theorem \ref{the7}, $u(x)$ is a locally semiconcave function. Furthermore, $\nabla^+ u(x)\neq \emptyset $ for every $x\in M$. Suppose $u$ is non-differentiable at $x$, then $\nabla^+ u(x)$ is not a singleton. By Lemma \ref{lem6}, $\nabla^* u(x)$ is not a singleton as well. Finally, using Lemma \ref{lem21} the curves satisfying condition \eqref{E17} must not be unique. This completes the proof.\qed
\end{proof}

\begin{proposition}\label{prop3}
For any $x\in M$, let $\gamma_x(t): [0, T)\rightarrow M$ be a future-directed timelike curve with $\gamma_x(0)=x$, and
\begin{equation}\label{E19}
  u(\gamma_x(t))=u(\gamma_x(0))+t~~\mbox{for every}~~t\in [0, T),
\end{equation}
  then $u$ is differentiable at $\gamma_x(t)$ and $\dot{\gamma}_x(t)=-\nabla u(\gamma_x(t))$ for any $t\in (0, T)$.
\end{proposition}
\begin{proof}
If there is a $T>s_0>0$ such that $u$ is not differentiable at $\gamma_x(s_0)$, then $\nabla^* u(\gamma_x(s_0))$ is not a singleton. By Lemma \ref{lem21}, there exists a future-directed timelike geodesic $\check{\gamma}_{\gamma_x(s_0)}:[s_0, T_1)\rightarrow M$ with $\check{\gamma}_{\gamma_x(s_0)}(s_0)=\gamma_x(s_0)$, $\dot{\check{\gamma}}_{\gamma_x(s_0)}(s_0)\neq\dot{\gamma}_x(s_0)$ and
\begin{equation}\label{E20}
  u(\check{\gamma}_{\gamma_x(s_0)}(s))=u(\gamma_x(s_0))+(s-s_0)~~\mbox{for every}~~s\in [s_0, T_1).
\end{equation}
 For the sake of convenience, let $\hat{\gamma}(t)=\gamma_x(t)$, for any $t\geq 0$,
$\check{\gamma}(s)=\check{\gamma}_{\gamma_x}(s)$ for any $s\geq s_0$.
Since $\dot{\hat{\gamma}}(s_0)\neq\dot{\check{\gamma}}(s_0)$. Then the curve $\check{\gamma}\ast(\hat{\gamma}|_{[0,s_0)}):[0, T_1)\rightarrow M$ has a corner at $\hat{\gamma}(s_0)$, where the $\ast$ denotes the conjunction of curves. By the transitivity of chronological relation $\ll$ \cite[p.55]{Beem 1996}, for $0<s_0 <s_1<T_1$, we have $x=\hat{\gamma}(0)\ll \hat{\gamma}(s_0) \ll\check{\gamma}(s_1)$. Since $M$ is globally hyperbolic, there is a maximal future-directed timelike geodesic $\sigma:[0, s_1]\rightarrow M$ such that $\sigma(0)=x$, $\sigma(s_1)=\check{\gamma}(s_1)$. By the smoothness of geodesic and reverse triangle inequality,
\begin{equation}\label{E21}
  d(\sigma(0), \sigma(s_1))=d(\hat{\gamma}(0), \check{\gamma}(s_1))>d(\check{\gamma}(s_0), \check{\gamma}(s_1))+d(\hat{\gamma}(0), \hat{\gamma}(s_0))=s_1.
\end{equation}
By Lemma \ref{lem1} and inequality (\ref{E21}), we obtain
\begin{align}\label{E22}
  u(\check{\gamma}(s_1))&\geq u(x)+d(x, \check{\gamma}(s_1)) \nonumber\\
                          &>u(x)+s_1.
\end{align}
On the other hand, by (\ref{E19}) and (\ref{E20}), we have
\begin{align}\label{E23}
  u(\check{\gamma}(s_1))&=u(\check{\gamma}(s_0))+(s_1-s_0) \nonumber \\
                          &=u(x)+s_0-(s_1-s_0)  \nonumber \\
                          &=u(x)+s_1.
\end{align}
This contradiction proves Proposition \ref{prop3}.\qed
\end{proof}

\textit{Proof of Theorem \ref{the2}.}
If $u$ is differentiable at $x$, by Proposition \ref{prop3}, $u$ is differentiable everywhere on $\gamma_x(t)$ for any $t\in [0, T)$. By subsection \ref{subs5.2}, we have
\[
\nabla u(\gamma_x(t))=-\dot{\gamma}_x(t)
\]
for any $t\in [0, T)$. This is to say, if $u$ is differentiable at $x$, there is only one future-directed timelike ray $\gamma_x:[0,T)\rightarrow M$ with $\gamma_x(0)=x$, $\dot{\gamma}(0)=-\nabla u(x)$ and
\[
u(\gamma(0))=u(\gamma(s))-s~~\mbox{for any}~~s\in [0,T).
\]
Together with Proposition \ref{prop2}, this impiles the first conclusion of Theorem \ref{the2}. The last conclusion of Theorem \ref{the2} follows from Proposition \ref{prop3}. Up to now, the proof of Theorem \ref{the2} is complete.

 In the end of this section, we will prove Theorem \ref{prop1}. There are two cases we will consider one by one.

\textit{Case 1}: The time orientations  of $\phi$ and $\varphi$ are identical.

Firstly, by Theorem \ref{the7}, both $\phi$ and $\varphi$ are locally semiconcave. By \cite[Proposition 1.1.3]{CannarsaPNDE2004}, we know $\min\{\phi , \varphi\}$ is locally semiconcave on $M$. Thus by Remark \ref{viscosity and smc}, we only need to show that $\min\{\phi , \varphi\}$ satisfies the equation (\ref{E1}) at its differentiable points. Let $M_1:=\{x:\phi(x) \neq \varphi(x)\}$ and $M_2:=\{x:\phi(x) = \varphi(x)\}$. Denote the interior of $M_2$ by $\text{int}(M_2)$. Let $U:=M_1\cup \text{int}(M_2)$, then $\min\{\phi , \varphi\}$ satisfies the equation (\ref{E1}) at its differentiable points in $U$. Obviously, $U$ is an open and dense subset, by \cite[Proposition 3.3.4]{CannarsaPNDE2004}, $\min\{\phi , \varphi\}$ satisfies the equation (\ref{E1}) at any differentiable point. So far, we know that $\min\{\phi , \varphi\}$ is really a viscosity solution to equation (\ref{E1}).

\textit{Case 2}: The time orientations  of $\phi$ and $\varphi$ are different.

Using the methods above we can show that $\max\{\phi , \varphi\}$ satisfies equation (\ref{E1}) at its differentiable points. Recall that $M_1=\{x:\phi(x) \neq \varphi(x)\}$ and $M_2=\{x:\phi(x) = \varphi(x)\}$. Then we only need to show that $\max\{\phi , \varphi\}$ satisfies the definition of viscosity solution at $M_2$.
Without loss of generality, we can assume that the time orientation of $\phi$ is future-directed and the time orientation of $\varphi$ is past-directed.
For any $x\in M_2$, it is easy to see that $\max\{\phi , \varphi\}$ is not differentiable at $x$. By Lemma \ref{lem11}, for any future-directed causal curve $\gamma:(-T, T):\rightarrow M$ with $\gamma(0)=x$, we have $\max\{\phi , \varphi\}$ strictly  decreases along $\gamma(t)$ for  $-T<t<0$ and increases along $\gamma(t)$ for $0<t<T$.

Firstly, we will show that $\nabla^+ \max\{\phi , \varphi\}(x)=\emptyset$ for $x\in M_2$. Actually, for any $V\in \nabla^+ \max\{\phi , \varphi\}(x)$, by the definition of $\nabla^+$, we can find a neighborhood $O$ of $x$ and a function $\psi\in C^\infty(O)$ with $\max\{\phi , \varphi\}(y)\leq \psi(y)$ in $O$, $\max\{\phi , \varphi\}(x)= \psi(x)$ and $\nabla \psi(x)=V$. Consequently, for any given future-directed causal curve $\gamma:[-T, T]\rightarrow M$ with $\gamma(0)=x$, we have
\[
\psi(y)\geq \max\{\phi , \varphi\}(y)\geq\max\{\phi , \varphi\}(x)= \psi(x)
\]
for any $y\in O\cap\gamma(t)$. In other words, $\psi(x)$ is a local minimum along $\gamma(t)$. This means that $g(\nabla \psi(x), \dot{\gamma}(0))=0$ for any future-directed causal curve $\gamma$ with $\gamma(0)=x$. Thus, $\nabla \psi(x)=0$. Recall that there is a forward calibrated curve $\gamma_x:[0,T)\rightarrow M$ with respect to $\varphi$ such that
\[
\varphi(\gamma_x(t))- \varphi(\gamma_x(0))=t
\]
for any $t\in [0, T)$. By Lemma \ref{lem11}, we know the above equality holds true for $\max\{\phi , \varphi\}$, i.e.,
\[
\varphi(\gamma_x(t))- \varphi(\gamma_x(0))=\max\{\phi , \varphi\}(\gamma_x(t))- \max\{\phi , \varphi\}(\gamma_x(0))=t
\]
for any $t\in [0, T)$. Note that $\psi\circ\gamma_x$ is smooth function from $[0,T)$ to $R$. By the Taylor formula, when $t$ is sufficiently small, we have
\begin{align*}
t=\max\{\phi , \varphi\}(\gamma_x(t))- \max\{\phi , \varphi\}(\gamma_x(0))\leq \psi(\gamma_x(t))-\psi(\gamma_x(0))=d \psi(\dot{\gamma}_x(0))t+o(t^2)= o(t^2).
\end{align*}
This contradiction means that $\nabla^+ \max\{\phi , \varphi\}(x)=\emptyset$.

Therefore, to obtain the second conclusion of Proposition \ref{prop1}, it suffices to show that if a vector $V\in \nabla^- \max\{\phi , \varphi\}(x)$, then $g(V, V)\geq -1$.

Without losing generality, let us assume that there exists a past-directed timelike vector $\overline{V}$ in $\nabla^- \max\{\phi , \varphi\}(x)$ such that
\[
g(\overline{V}, \overline{V})<-1.
\]
We fix a suitable neighborhood $O$ of $x$, by the definition of $\nabla^-$, we can find a smooth function $\zeta\in C^\infty(O)$ such that $\zeta(y)\leq \max\{\phi , \varphi\}(y)$ for any $y\in O$, the equality holds at $y=x$ and moreover $\nabla \zeta(x)=\overline{V}$. By the similar analysis as above, for a forward calibrated curve $\gamma_x:[0,T)\rightarrow M$ with respect to $\varphi$, we have
\[
\max\{\phi , \varphi\}(\gamma_x(t))- \max\{\phi , \varphi\}(\gamma_x(0))=t
\]
for any $t\in [0, T)$. By the Taylor formula and wrong-way Cauchy-Schwarz inequality, when $t$ is sufficiently small, we have
\begin{align*}
t=\max\{\phi , \varphi\}(\gamma_x(t))- \max\{\phi , \varphi\}(\gamma_x(0))&\geq \zeta(\gamma_x(t))-\zeta(\gamma_x(0))\\
                                                                          &=d \zeta(\dot{\gamma}_x(0))t+o(t^2)\\
                                                                          &\geq |\overline{V}||\dot{\gamma}_x(0)|t+o(t^2)\\
                                                                          &= |\overline{V}|t+o(t^2).
\end{align*}
This contradiction indicates that if $V\in \nabla^- \max\{\phi , \varphi\}(x)$, then $g(V, V)\geq -1$.  Theorem \ref{prop1} is proved.         \qed


\section{Characterization of $C_u$}\label{sec:6}

Up to now, we have discussed the case when the time orientation of the viscosity solution is consistent. Our results show that when the time orientation of viscosity solution is consistent, the properties of the viscosity solution are essentially not much different from the case in Riemannian manifold. But when the time orientation of viscosity solution is non-consistent, viscosity solutions have many peculiar properties.

At the beginning of this section, we provide an example, which can illustrate the differences of viscosity solutions in the case of Riemannian manifold and space-time.

\textit{Example}. Let $(M, g)$ be the 2-dimensional Minkowski space-time, and $f(x,y)= \sqrt{2}|x|-|y|$. Clearly, $f$ is Lipschitz on $M$. By a simple calculation,
\begin{align*}
\nabla f= \left\{
\begin{array}{ll}
\displaystyle (\sqrt{2},  1)~~~~~~~&x>0, y>0,\\
\displaystyle (\sqrt{2}, -1)~~~~~&x>0, y<0,\\
\displaystyle (-\sqrt{2},  1)~~~~~&x<0, y>0,\\
\displaystyle (-\sqrt{2}, -1)~~~~~&x<0, y<0.
\end{array}\right.
\end{align*}
It is easy to check that $f$ is indeed a globally defined viscosity solution to equation (\ref{E1}). Note that $\nabla^+ f((0, 0))=\nabla^-f ((0, 0))=\emptyset$. See the Figure 2.

\begin{figure}\label{P1}
\begin{center}
\includegraphics[scale=0.3]{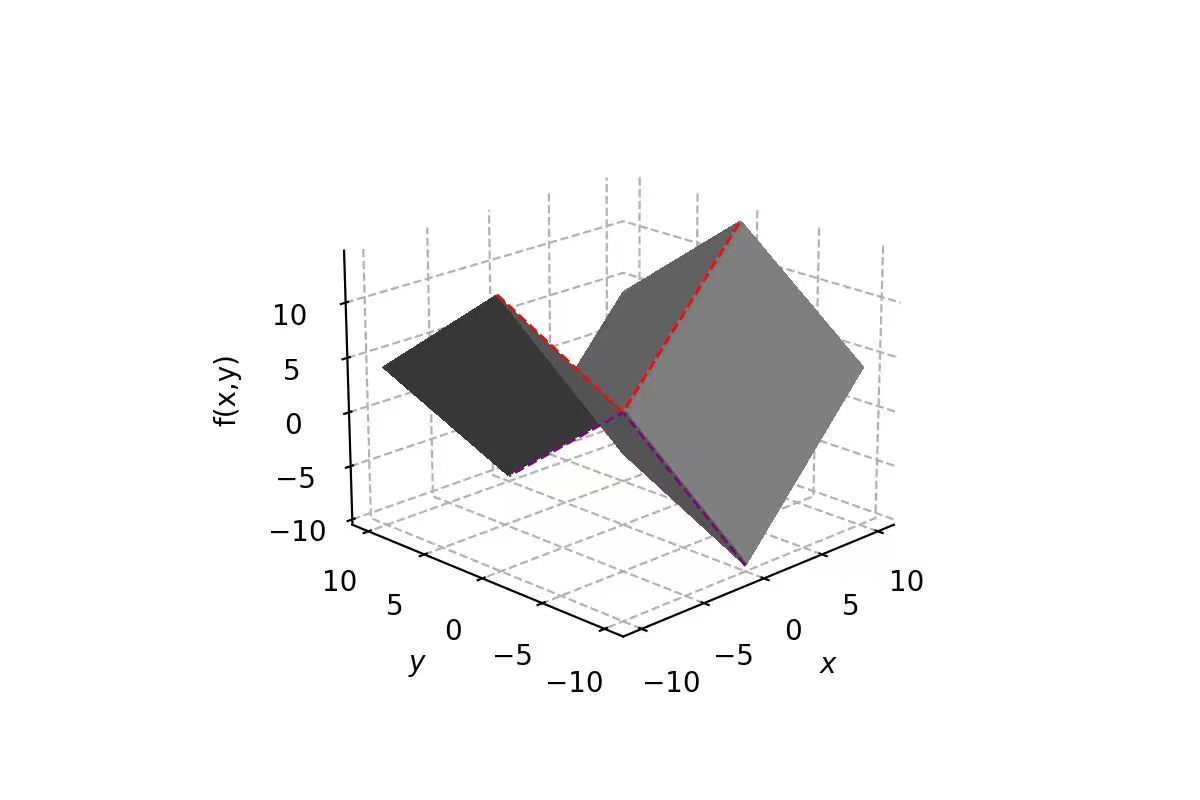}
\caption{A globally defined viscosity solution to equation (\ref{E1}) on the 2-dimensional Minkowski space-time.}
\end{center}
\end{figure}

\begin{lemma}\label{lem26}
For $u\in \mathcal{S}(M)$ and any $x\in M\setminus C_u$, there exists a neighborhood $U_x$ of $x$ such that the time orientation of $u$ is consistent on $U_x$.
\end{lemma}
\begin{proof}
Without loss of generality, we assume that the time orientation of $u$ at $x$ is past-directed. If the conclusion is not true, then there exists a sequence $\{x_i\}$ converging to $x$ and a future-directed timelike vector $V_i\in \nabla^* u(x_i)$ for each $i$. By the definition of $\nabla^*$, for each $i$ there exists a sequence $x_i^n\rightarrow x_i$ as $n\rightarrow\infty$ such that $u$ is differentiable at every $x_i^n$ and $\nabla u(x_i^n)\rightarrow V_i$ as $n\rightarrow\infty$. By the diagonal method, we can select a subsequence $x_{n_i}$ satisfying $x_{n_i}\rightarrow x$ as $i\rightarrow\infty$. Moreover,  $u$ is differentiable at $x_{n_i}$ and $\nabla u(x_{n_i})$ is future-directed. Then the future-directed timelike vector $\lim\limits_{i\rightarrow\infty}\nabla u(x_{n_i})$ must belong to $\nabla^* u(x)$. This contradicts the time orientation of $u$ at $x$ is past-directed. \qed
\end{proof}

\textit{Proof of Theorem \ref{the5}.} The first three conclusions of Theorem \ref{the5} are immediate consequences of Lemma \ref{lem26}. Actually, since $u\in \mathcal{S}(M)$, $u$ is differentiable almost everywhere. This, together with Lemma \ref{lem26}, implies that $ M\setminus C_u$ is an open dense subset of $M$. Assume that $C_u\neq\emptyset$ and $M\setminus C_u$ is connected. For any $x\in C_u$, by the definition of $C_u$, there exist two differentiable points $y, z$ and a suitable neighborhood $U$ of $x$ such that $y,z\in U$ and the time orientation of $u$ at $y$ is different from the time orientation of $u$ at $z$. Obviously, $y, z\in M\setminus C_u$. Since $M\setminus C_u$ is connected, there exists a continuous curve $\alpha:[0, 1]\rightarrow M$ such that $\alpha(0)=y$, $\alpha(1)=z$ and $\alpha|_{[0,1]}\subseteq M\setminus C_u$. By Lemma \ref{lem26}, for each $0\leq t \leq 1$, there exists a neighborhood $U_{\alpha(t)}$ of $\alpha(t)$ such that the time orientation of $u$ is consistent on $U_{\alpha(t)}$. Since $\alpha|_{[0,1]}$ is compact and $M\setminus C_u$ is  open, we can choose a family of finite neighborhoods $\{U_{\alpha(t_i)}\}_i$ such that $\bigcup\limits_{i}U_{\alpha(t_i)}\subseteq M\setminus C_u$ and the time orientation of $u$ is consistent on $U_{\alpha(t_i)}$. This means that the time orientation of $u$ at $y$ is consistent with the time orientation of $u$ at $z$. This is a contradiction.

Note that $C_u$ is a subset of $\{x\in M|~ \text{$u$ is not differentiable at $x$}\}$, then $\mathcal{M}(C_u)=0$. Recall that an open set $U$ in a space-time is said to be causally convex if no causal curve intersects $U$ in a disconnected set. Since $(M, g)$ is globally hyperbolic, $(M, g)$ is strongly causal and consequently for any $x\in M$, $x$ has arbitrarily small casually convex neighborhoods. Then the open neighborhood $U_x$ in Lemma \ref{lem26} can be chosen to be a causally convex neighborhood. Obviously, $(U_x, g|_{U_x})$ is causal since $(M, g)$ is causal. For any $y, z\in U_x$ with $y\leq z$, all future-directed causal curves from $y$ to $z$ must stay in $U_x$, i.e., $J^+(y)\cap J^-(z) \subseteq U_x$ if $y\leq z$. Since $J^+(y)\cap J^-(z)$ is compact in $M$, $J^+(y)\cap J^-(z)$ is compact with respect to the induced topology in $U_x$.  In a word, the space-time $(U_x, g|_{U_x})$ is globally hyperbolic. By Lemma \ref{lem26}, the time orientation of $u$ is consistent on $(U_x, g|_{U_x})$. Therefore the third conclusion of  Theorem \ref{the5} is clearly valid.

In the following, we will prove  the last conclusion of Theorem \ref{the5}. we firstly claim that the time orientation of $u$ is consistent on $J^+(x)\setminus \{x\}$ and $J^-(x)\setminus \{x\}$ respectively. For any $y\in J^+(x)\setminus \{x\}$, since $C_u$ is an acausal set, thus $y \notin C_u$ and consequently, the time orientation of $y$ is either future-directed or past-directed. Assume that there exist $y_0, y_1\in J^+(x)\setminus \{x\}$ such that the time orientation of $u$ at $y_0$ is different from the time orientation of $u$ at $y_1$. By the path connectedness of $J^+(x)\setminus \{x\}$, there is a continuous curve $\beta: [0 ,1]\rightarrow J^+(x)\setminus \{x\} $ such that $\beta(0)=y_0$ and $\beta(1)=y_1$. Meanwhile, by Lemma \ref{lem26}, for each $0\leq t \leq 1$, there exists a neighborhood $U_{\beta(t)}$ of $\beta(t)$ such that the time orientation of $u$ is consistent on $U_{\beta(t)}$. By the same argument above, we can choose a family of finite neighborhoods $\{U_{\beta(t_i)}\}_i$ such that $\bigcup\limits_{i}U_{\beta(t_i)}\subseteq M\setminus C_u$ and the time orientation of $u$ is consistent on $U_{\beta(t_i)}$. This means that the time orientation of $u$ at $y_0$ is consistent with the time orientation of $u$ at $y_1$, which contradicts our assumption. Then, the time orientation of $u$ is consistent on $J^+(x)\setminus \{x\}$. Similarly, the time orientation of $u$ is consistent on $J^-(x)\setminus \{x\}$.

Assume that $\text{edge} (C_u)\neq\emptyset$. By the definition of $\text{edge}$, for $x\in \text{edge} (C_u)$ and any causal convex neighborhood $U$ of $x$, there exists a timelike curve $\gamma$ from $I^-(x, U)$ to $I^+(x, U)$ such that $\gamma\cap C_u=\emptyset$. By the similar analysis above, the time orientation of $u$ is consistent on $I^-(x)\cup I^+(x)$. Recall that $C_u$ is a closed subset of $M$, then $x\in \text{edge} (C_u)$ implies $x\in C_u$.  By the definition of $C_u$, for any neighborhood $U_1$ of $x$, there exist points $y_2, y_3\in U_1\setminus C_u$ such that the time orientation of $u$ at $y_2$ is different from the time orientation of $u$ at $y_3$. We fix $z_1\in I^-(x, U)$, $z_2\in  I^+(x, U)$ and choose $U_1$ such that $U_1 \subseteq J^+(z_1)\cap J^-(z_2)$. Then by  the transitivity of chronological relation $\ll$, there exist timelike curves $\gamma_2, \gamma_3$ from $I^-(x, U)$ to $I^+(x, U)$ which passes through $y_2, y_3$ respectively. Since $C_u$ is an acausal set, $\gamma_2$ intersects $C_u$ at most one point. Moreover, $y_2 \notin C_u\cap \gamma_2$. Then the time orientation of $u$ at $y_2$ is consistent with  the time orientation of $u$ on either $I^-(x)$ or $I^+(x)$. Similarly, the time orientation of $u$ at $y_3$ is consistent with  the time orientation of $u$ on either $I^-(x)$ or $I^+(x)$. Since the time orientations of $u$ are the same as on $I^-(x)$ and $I^+(x)$, $y_2, y_3$ must belong to the same connected component of $M\setminus C_u$. This contradicts the selection of $y_2$ and $y_3$. Thus, $\text{edge} (C_u)=\emptyset$ and consequently $C_u$ is a partial Cauchy surface. The proof of Theorem \ref{the5} is complete.

\begin{remark}
Applying Theorem \ref{the5} to the function $\max\{\phi,\varphi\}$ in Theorem \ref{prop1}, it is easy to see that 
\[
C_{\max\{\phi,\varphi\}}=\{x\in M|~ \phi(x)=\varphi(x)\}
\]
is a partial Cauchy surface.
\end{remark}

\section*{Acknowledgments}

We sincerely thank Professor Jin Liang for the helpful advice.
Xiaojun Cui and Siyao Zhu are supported by the National Natural Science
Foundation of China (Grant No. 12171234).
Hongguang Wu is supported by the National Natural Science Foundation of China (Grant no. 12201073).

\end{document}